\documentclass[10pt,twoside]{article}
\usepackage{graphics}
\usepackage{amsmath, amsfonts, amsthm, amssymb, amscd,a4wide}
\usepackage[final,dvips]{graphicx}
\usepackage{hyperref}
\usepackage{url}
\usepackage{mathrsfs} 
\usepackage{booktabs,a4wide}

\usepackage{ulem}

\usepackage{epsfig, subfig}   
\numberwithin{figure}{section}

\usepackage{booktabs}
\usepackage{algorithm} 
\usepackage{algorithmic}
\usepackage{longtable,booktabs}

\usepackage{float}

\usepackage{caption}

 \usepackage{ulem,lipsum}
\newcommand\xoutpars[1]{\let\helpcmd\xout\parhelp#1\par\relax\relax}
\newcommand\soutpars[1]{\let\helpcmd\sout\parhelp#1\par\relax\relax}
\long\def\parhelp#1\par#2\relax{%
  \helpcmd{#1}\ifx\relax#2\else\par\parhelp#2\relax\fi%
}

\usepackage{perpage}
\MakePerPage{footnote} 
\usepackage{slashed}
\usepackage{xcolor}
\usepackage{hyperref}
\hypersetup{linktoc = all}                
\hypersetup{hidelinks}
\hypersetup{bookmarksnumbered}
\let\oldmarginpar\marginpar
\renewcommand\marginpar[1]{\-\oldmarginpar[\raggedleft\footnotesize #1]%
{\raggedright\footnotesize #1}}

\newtheorem{definition}{Definition}[section]

\newtheorem{lemma}[definition]{Lemma}

\newtheorem{proposition}[definition]{Proposition}
\newtheorem{claim}[definition]{Claim}

\numberwithin{equation}{section}
\numberwithin{table}{section}

\usepackage[most]{tcolorbox}

\newtcolorbox{markbox}{%
     enhanced, breakable, size=minimal, parbox=false, after={\par}, 
     before upper={\indent}, colback=white, 
     overlay = {\draw[line width=2pt] (frame.north east) -|
                       ([xshift=3mm]frame.east)|-(frame.south east);},
     overlay first={\draw[line width=2pt] (frame.north east) -|
                           ([xshift=3mm]frame.south east);},
     overlay middle={\draw[line width=2pt] ([xshift=3mm]frame.north east) -- 
                              ([xshift=3mm]frame.south east);},
     overlay last={\draw[line width=2pt] ([xshift=3mm]frame.north east)|-
                          (frame.south east);},
}

\usepackage[most]{tcolorbox} 
\newenvironment{Figure}
  {\par\medskip\noindent\minipage{\linewidth}}
  {\endminipage\par\medskip}

\newtcolorbox{done}{%
     enhanced, breakable, size=minimal, colframe=white, parbox=false, after={\par\vspace{2\baselineskip}}, 
     before upper={\indent}, colback=white, 
     overlay = {\draw[line width=2pt] (frame.north east) -|
                       ([xshift=3mm]frame.east)|-(frame.south east);},
     overlay first={\draw[line width=2pt] (frame.north east) -|
                           ([xshift=3mm]frame.south east);},
     overlay middle={\draw[line width=2pt] ([xshift=3mm]frame.north east) -- 
                              ([xshift=3mm]frame.south east);},
     overlay last={\draw[line width=2pt] ([xshift=3mm]frame.north east)|-
                          (frame.south east);},
}
\newtcolorbox{adone}{%
     enhanced, breakable, size=minimal, colframe=white, parbox=false, after={\par\vspace{2\baselineskip}}, 
     before upper={\indent}, colback=white, 
     overlay = {\draw[densely dotted, line width=2pt] (frame.north east) -|
                       ([xshift=3mm]frame.east)|-(frame.south east);},
     overlay first={\draw[densely dotted, line width=2pt] (frame.north east) -|
                           ([xshift=3mm]frame.south east);},
     overlay middle={\draw[densely dotted, line width=2pt] ([xshift=3mm]frame.north east) -- 
                              ([xshift=3mm]frame.south east);},
     overlay last={\draw[densely dotted, line width=2pt] ([xshift=3mm]frame.north east)|-
                          (frame.south east);},
}
\newtcolorbox{plfok}{%
     enhanced, breakable, size=minimal, colframe=white, parbox=false, after={\par\vspace{2\baselineskip}}, 
     before upper={\indent}, colback=white, 
     overlay = {\draw[dashed, line width=2pt] (frame.north east) -|
                       ([xshift=3mm]frame.east)|-(frame.south east);},
     overlay first={\draw[dashed, line width=2pt] (frame.north east) -|
                           ([xshift=3mm]frame.south east);},
     overlay middle={\draw[dashed, line width=2pt] ([xshift=3mm]frame.north east) -- 
                              ([xshift=3mm]frame.south east);},
     overlay last={\draw[dashed, line width=2pt] ([xshift=3mm]frame.north east)|-
                          (frame.south east);},
}

\newcommand \bei {\begin{itemize}}
\newcommand \eei {\end{itemize}}

\newcommand \Span {\text{Span}}

\newcommand \be   {\begin{equation}}
\newcommand \bel {\be\label}
\newcommand \ee   {\end{equation}}

\newcommand \bse {\begin{subequations}}
\newcommand \ese {\end{subequations}}

\newcommand \supp {{\text{supp }}}
\newcommand \Id   {{\text{Id}}}

\newcommand \RR    {\mathbb{R}}

\newcommand \EE    {\mathbb{E}}

\newcommand \Tbar {\overline T}
\newcommand \Acal {\mathcal A}

\newcommand \Hcal    {\mathcal H}

\newcommand \Lcal    {\mathcal{L}}

\newcommand \Wbf {\mathbf W}

\newcommand \Ncal    {\mathcal{N}}

\newcommand \RN    {{\RR^N}}
\newcommand \eps   {\epsilon}

\newcommand \Rd    {{\mathbb{R}^d}}

\newcommand \del   \partial

\newcommand \la         \langle
\newcommand \ra     \rangle

\newcommand \plus {+}

\newcommand \RD {{{\mathbb R}^D}}

\newcommand \Ybar { \widetilde Y} 
 
\newcommand \dK {{\mathbf d}^K} 
\newcommand \Ebf {\mathbf E} 
\newcommand \ebf {\mathbf e} 

\newcommand \Pbf {\mathbf P} 
\newcommand \Extra {\textbf{Ext}} 
\newcommand \Inter {\textbf{Int}} 

\newcommand  \Pibf {\mathbf \Pi}

\newcommand \muexa {\mu^{\text{ex}}} 
\newcommand \exact {\text{ex}} 

\begin{document}

\bibliographystyle{plain}

\title{  A class of mesh-free algorithms for some problems 
\\
arising in finance and machine learning
} 

\author{Philippe G. LeFloch\footnote{Laboratoire Jacques-Louis Lions, Centre National de la Recherche Scientifique, Sorbonne Universit\'e, 4 Place Jussieu, 75252 Paris, France. 
Email: {\it contact@philippelefloch.org}.}
 \, and Jean-Marc Mercier\footnote{MPG-Partners, 136 Boulevard Haussmann, 75008 Paris, France. 
 Email: {\it jean-marc.mercier@mpg-partners.com.}
\qquad 
\newline  
\hfill To appear in: {\it Journal of Scientific Computing} (2023). 
}}
\date{}

\maketitle

\begin{abstract} 

{
We introduce a numerical methodology,
 referred to as the {\sl transport-based mesh-free method}, 
 which allows us to deal with continuous, discrete, or statistical models in the same unified framework,
 and leads us to a broad 
class of numerical algorithms recently implemented in a Python library (namely, CodPy). 
Specifically, we propose a mesh-free discretization technique based on the theory of reproducing kernels and the theory of transport mappings, in a way that is reminiscent of Lagrangian methods in computational fluid dynamics. We introduce kernel-based
discretizations of a variety of differential and discrete operators (gradient, divergence, Laplacian, Leray projection, extrapolation, interpolation, polar factorization). The proposed algorithms are nonlinear in nature and enjoy quantitative error estimates based on the notion of discrepancy error, which allows one to evaluate the relevance and accuracy of, both, the given data and the numerical solutions.  
Our strategy is relevant when a large number of degrees of freedom are present as is the case in mathematical finance and machine learning. We consider the Fokker-Planck-Kolmogorov system (relevant for problems arising in finance and material dynamics) and a class of neural networks based on support vector machines. 
}
\end{abstract}

  
  \vfill

\setcounter{secnumdepth}{3}
\setcounter{tocdepth}{1} 
\tableofcontents
 
 \vfill

\newpage 
\section{Introduction}

{

\paragraph*{Background.}

We are interested in various applications in which mesh-free methods provide a flexible approach to design efficient computational methods. The kernel-based methodology advocated in the present paper, in the form described below, is motivated by earlier works in {  computational science} and statistics. We intend here to provide a unified framework which, importantly, also includes quantitative error estimates and allows us to evaluate the relevance and accuracy of, both, the given data and the numerical solutions. 

For a background on mesh-free approaches, we refer the reader to the textbooks \cite{BTA} and \cite{LiLiu} on the theory of reproducing kernels and their applications in computational statistics and fluid dynamics, respectively. In \cite{BBO,Liu-2003} a review is provided from the standpoint for generalized finite element methods.  
Designing algorithms that merge together techniques of machine learning and numerical methods for partial differential equations is an active domain of research; see~\cite{Haghighat,Liu-2016} for applications in computational solid mechanics, \cite{ZhouLi}
for applications in phase field dynamics, and \cite{BFBL} for the standpoint of peri-dynamics. In 
 \cite{Sirignano} a general treatment of partial differential equations is presented while in \cite{Koester} a so-called conforming window technique is introduced in order to handle non-convex computational domains. Many contributions in approximation theory are also available and we refer to the reviews \cite{FGE, Wendland,Wendland-book} and to \cite{Salehi} for the analysis of a moving mesh approach. 

  
\paragraph*{Main purpose of this paper.}

The numerical methodology proposed here is mesh-free and is based on a combination of analytical and algorithmic techniques, including the theory of reproducing kernels and the theory of transport mappings. In comparison to earlier works, including our earlier investigations~\cite{PLF-JMM-1}--\cite{PLF-JMM-estimate}, our framework has the advantage of treating 
many applications of interest simultaneously 
and covers a wide class of differential and discrete operators, including projection operators,
 extrapolation/interpolation operators, and the polar decomposition of maps. 
 The {\sl transport-based mesh-free method}, as we call it, 
applies to problems arising, for instance, in data science or statistics, 
 and is relevant when a large number of degrees of freedom are present. Together with earlier works on the subject, 
 we contribute here to tackling the curse of dimensionality, namely the phenomenon that limits the efficiency of numerical algorithms when analyzing data in high-dimensional spaces. 

 After introducing elementary material about the theory of reproducing kernels 
which serves as a backbone for the development of numerical algorithms, we present (cf.~Section 2 below) kernel-based discretization formulas for differential operators and extrapolation/interpolation operators. By using such operators, we can guarantee that our numerical algorithms are consistent\footnote{ By consistency, we mean that the schemes can only converge to a solution to the given problem.}
{ 
 with the problems under consideration and,  for several concrete applications at least, 
  provide efficient tools\footnote{We do not perform here comparisons with existing algorithms (a task tackled in~\cite{PLF-JMM-SM-long}).}
 as demonstrated by our numerical tests. The theoretical analysis of the algorithms is outside the scope of the present paper, which is focused on the presentation of the techniques and applications.

More precisely, in this paper we consider three domains of application of independent interest, as follows. 
\bei 

\item Machine learning (Section 3): we revisit the formulation of Support Vector Machine arising in data science, and formulate an algorithm that combines together three levels and possibly several kernels.  {  While only preliminary illustrations are provided here, we refer the reader to \cite{PLF-JMM-SM-long} for specific applications to supervised and unsupervised machine learning. } 

\item Mathematical finance (Section 4): we introduce a kernel-based discrete scheme for the approximation of the (forward) Fokker-Planck equation and the (backward) Kolmogorov equation, respectively. 
{  
This strategy was applied in a concrete problem arising in finance engineering, for instance for the analysis of hedging strategies for net interest income and economic values of equity, as reported in \cite{PLF-JMM-review,JMM-SM}.  }

\item Polar decompositions (Section 5): we revisit the classical problem of the (signed) polar decomposition of mappings, and propose continuous and discrete algorithms.

\eei 

}

\paragraph*{  Features of the method.}

For each problem we rely on our notions (Section~\ref{section--2}) of kernel-based discretization, interpolation, and extrapolation. 
The discrete solutions depend upon the selection of a reproducing kernel, {  which should be chosen on a case-by-case basis depending upon the application of interest; for a preliminary discussion of this question, we refer the reader to our earlier study \cite{PLF-JMM-estimate} in which a comparison between different kernels is given.}
{  
Importantly, the proposed algorithms enjoy stability properties in the form of (a posteriori) error estimates which are based on the notion of discrepancy distance and discrepancy sequences analyzed earlier in \cite{PLF-JMM-estimate}.}
{
We emphasize that our schemes are nonlinear in nature, even when the problems under consideration are linear. This feature is well-known in fluid dynamics where for advection equations as well as contact discontinuities (which are linear waves in nature) are treated with nonlinear schemes, for instance with Harten's artificial compression method \cite{Sod}.}

Each algorithm is illustrated by numerical experiments which demonstrate the efficiency of the proposed approach. While the present paper describes the algorithms and contains only basic numerical tests, { the companion papers }
\cite{PLF-JMM-SM-Tutorial}--\cite{PLF-JMM-SM-long} are devoted to the implementation of our algorithms in a Python code (called CodPy, standing for the ``Curse of dimensionality in Python'').

The use of the reproducing kernel techniques, as advocated here and in \cite{PLF-JMM-1}--\cite{PLF-JMM-estimate}, 
has several motivations. An important feature implied by the use of reproducing kernels is the possibility of computing error estimates based on the notion of discrepancy distance. 
Indeed, this approach allows for a {\sl quantitative evaluation}
of the integration error or the discretization error associated with a given set of data and a numerical solution. 
This is a significant advantage in comparison to standard approaches used 
in the theory of neural networks and in mathematical finance. 
{  In fact, this is the main advantage of our approach.}
Error estimates are essential in applications in order to decide the relevance of 
a numerical output, and also provide 
an important feedback on the relevance of the mathematical model itself. {  
In addition, 
kernel engineering techniques fit well within such a framework, 
that is the construction of ``new'' kernels from sums, products, or compositions of given kernels.} 

{ 
We build upon our earlier study \cite{PLF-JMM-estimate} in which we numerically investigated the kernel-based mesh-free approximation of multi-dimensional integrals. }
A kernel typically captures regularity and qualitative properties of functions ``beyond'' the standard Sobolev regularity class, which is an essential flexibility for many problems under consideration in data science and
in finance. In  \cite{PLF-JMM-estimate}, we made comparisons between several numerical strategies, several choices of kernels, and we 
investigated the role played by the number of discretization points and the dimension of the problem. 
In short, quantitative error estimates are available for kernel-based algorithms, {
 in contrast to random algorithms which do not enjoy such deterministic estimates.}

\paragraph*{Outline of this paper.}

{

In Section~\ref{section--2}, 
to any given reproducing kernel, we associate mesh-free formulas for the discretization of integrals and differential operators; in this framework we also introduce discrete interpolation/extrapolation operators.  
Section~\ref{sec-trois} is devoted to their {  application to support vector machines, which play a key role
 in data science.} 
In Section~\ref{sec-44}, we introduce an algorithm for the Fokker-Plank-Kolmogorov system,
 which 
then allows us to tackle 
typical problems arising in mathematical finance. 
Finally, in 
Section~\ref{sec-55} we propose an algorithm for the signed-polar decomposition which is relevant in transportation theory. 
}

$$
\aligned
&   && \text{\bf Table of notation}
\\
& \Pbf^K_Y(f) && \text{ $K$-discrete projection} && \eqref{equa-216}
\\
& \nabla^K_Y             && \text{$K$-gradient operator}  && \eqref{Grad}
\\
& (\nabla^K_Y)^T   && \text{$K$-divergence operator} && \eqref{equa-diverg} 
\\
& \Extra_{X,Y}^K f_X  && \text{$K$-extrapolation operator} && \eqref{equation:EX-00}
\\
& \Inter_{X,Y}^K f_X  && \text{$K$-interpolation operator} && \eqref{IN}
\\
&  \dK(\mu_Y, \mu_X)  && \text{discrepancy error}  && \eqref{EEPO}
\endaligned
$$

}


\section{Mesh-free discretization of integrals and differential operators}
\label{section--2} 

\subsection{Discrepancy and transport}

\paragraph*{Error integration estimate.}

We begin with an informal discussion restricted to the whole space $\RD$ (with $D \geq 1$), before presenting our discretization formulas for differential operators and extrapolation/interpolation operators. 
Introducing mesh-free, kernel-based discretization formulas is our first objective, since they will be at the heart of the algorithms of the following three sections. 
Observe that sufficient regularity is assumed throughout this paper, and we do not discuss any issue related to mathematical analysis. 

{ Throughout, we work with unstructured meshes in $\RD$, each being represented by a set\footnote{ The points are arbitrary at this stage while later they will be chosen in order to minimize an error functional such as \eqref{equation-EE} below.}
 of $N$ distinct points $Y = (y^1,\ldots,y^N)$. 
Specifically, given probability measure $\mu$ on $\RD$ we consider an approximation 
 by a sum of Dirac masses $\delta_{y^n}$ with equal weights (the choice of points depending upon the measure), that is, }
$
\mu \simeq \delta_Y  = (1/N)\sum_{1 \leq n \leq N} \delta_{y^n}. 
$
The accuracy of this approximation is determined by a distance among measures for averages of continuous and globally integrable functions $\varphi: \RD \to \RR$.  To this end, a functional space denoted by $\Hcal^K(\RR^D)$ is defined and, in this space, we seek an estimate of the form\footnote{ In practice, the error estimate should also take into account the are approximation errors, in the form of a variance term taking into account the variability of random data.} 
\bel{equation:MCE}
\Big| \int_{\RR^D} \varphi \, d\mu - \frac{1}{N} \hskip-.1cm
\sum_{1 \leq n \leq N} \varphi(y^n) \Big| 
\leq \dK(\mu, \delta_Y) \|\varphi\|_{\Hcal^K(\RR^D)}. 
\ee
In the setup under consideration, the optimal coefficient $\dK(\mu, \delta_Y)$ can be expressed explicitly in terms of the kernel and the set of points.


\paragraph*{Discrepancy error.}

More precisely, $\Hcal^K(\RR^D)$ is the so-called reproducing-kernel Hilbert space associated with a given kernel $K$. By definition, the latter is a continuous and symmetric function $K: (x,y) \in \RD \times \RD \mapsto \RR$ such that the matrix 
$K(Y,Y)= \big( K(y^n,y^m) \big)_{1 \leq n,m \leq N}$ is (symmetric) positive-definite for any set of distinct points.  
The Hilbert space of interest $\Hcal^K(\RR^D)$ defined by completion of $\Span \big\{K(\cdot, x) \, / \, x \in \RR^D \big\}$ consists of all linear combinations of the functions $K(x, \cdot)$ parametrized by $x \in \RR^D$.
This space is endowed with a scalar product (in \eqref{equa-538} below), and the coefficient $\dK(\mu, \delta_Y)$ in \eqref{equation:MCE} is referred to as the {\sl discrepancy error} (investigated earlier by the authors in \cite{PLF-JMM-estimate}) and reads 
\bel{equation-EE}
\aligned
   & \dK(\mu, \delta_Y)^2
 = \iint_{\RD \times \RD} K(x,y)d\mu(x)d\mu(y) 
  + \frac{1}{N^2} \hskip-.1cm
\sum_{n,m=1}^N \hskip-.2cm
K(y^n,y^m)- \frac{2}{N} \sum_{1 \leq n \leq N} \int_{\RR^D} K(x,y^n)d\mu(x). 
\endaligned
\ee
A set of points $\Ybar$ is called a {\sl sharp discrepancy set} if it achieves the global minimum of the functional, that is, $\Ybar = \arg \inf_Y \dK(\mu,\delta_Y)$. 
{
Even when the minimum is not achieved exactly, }
in numerical applications we seek a numerical approximation $\Ybar$ and use $\dK(\mu, \delta_{\Ybar})$
as our discrepancy error in \eqref{equation:MCE}. 
This is the setup in which we can evaluate the error made in our algorithms below. 


\paragraph*{Methodology and transport maps.}

An additional feature of our approach is as follows.   
{ 
In a given application we choose an admissible kernel $K$ and introduce the corresponding functional space $\Hcal^K(\RD)$. We then have various discrete approximation formulas as described in the rest of this section,}
 whose accuracy is determined by computing the error function 
in \eqref{equation:MCE}--\eqref{equation-EE}. Then, in order to optimize the convergence rate we can compute numerically an approximation of a sequence of points.
From a practical point of view, the discrepancy error can computed quite efficiently by 
a direct Monte-Carlo approach as we do in the tests presented in this paper. We thus have a concrete approach for evaluating the accuracy of our numerical solutions in a {\sl quantitative} way, which is essential in many applications.  

{ 
We can also apply an optimal transport step}
 \cite{Villani}, as follows. Let us consider a convex and open set $\Omega \subset \RD$ with piecewise smooth boundary {
 which, in the applications, we typically choose to be $(0,1)^D$.
 } 
Using transportation,  we reduce the study of \eqref{equation:MCE} to the same problem for the Lebesgue measure $\lambda = dx$ on $\Omega$. 
Namely,  consider the transport map $S:\Omega \mapsto \RR^D$ associated with a given measure $\mu$
(with convex support), that is, the unique map satisfying 
$
\int_{\RR^D} \varphi d\mu = \int_{\Omega} (\varphi \circ S)dx, 
$
for all continuous and globally integrable functions $\varphi : \RD \to \RR$ 
which moreover takes the form $S = \nabla h$ for some convex maps $h$ ($\nabla$ denoting the gradient operator). 
The inequality \eqref{equation:MCE} is then equivalent to saying 
\bel{equation:MCEL}
\Big| \int_{\Omega} (\varphi \circ S) \, dx - \frac{1}{N}\sum_{1 \leq n \leq N} (\varphi \circ S)(x^n) \Big| 
\leq \dK(\lambda, \delta_{S^{-1} Y}) \|\varphi \circ S\|_{\Hcal^K(\Omega)},
\ee
{
where 
 $Y=(y^n) = (S(x^n))$
 and $\varphi \circ S$ denotes the composition of two maps. Here,
  $\Hcal^K(\Omega)$ denotes the space associated with the set $\Omega$ (see below). 
  }
 

\subsection{Notation for the continuous framework}
\label{section-210}

\paragraph*{Admissible kernels.}

{
We are primarily interested in kernels defined on a bounded and convex set and, from now on, 
we focus our presentation mainly to this class (although we will allow ourselves to manipulate kernels defined on the whole of $\RD$). More precisely, while a prescribed kernel might be originally defined on $\RR^D$, 
then via a composition with a map $S$ we construct kernels $(k \circ S)(x,y)$ that have compact support. 
}
A reproducing kernel \cite{FGE,Wendland} provides a convenient way to generate a broad class of Hilbert spaces. 
A bounded and continuous function $K: \Omega \times \Omega \to \RR$ on a bounded open set  $\Omega \subset \RD$ 
is called an {\it admissible kernel} if it satisfies the 
 (1)  {\it symmetry property} $K(x,y) = K(y,x)$ for all $x, y \in \Omega$,   
{ and the
following 
(2) }{\it positivity property:} for any collection of $N$ distinct points $Y = (y^1, \ldots, y^N)$ in $\Omega$, the symmetric matrix $K(Y,Y) = \big( K(y^m,y^n) \big)_{1 \leq n,m \leq N}$ is positive definite in the sense that $a^T K(Y,Y) a > 0$ for all $a \in \RN \setminus \{0\}$. 
It is said to be {\it uniformly positive} if there exists a uniform constant $c>0$ such that for any collection of distinct points $Y$ one has $a^T K(Y,Y) a \geq c \, |a|^2$ for all $a \in \RN$.

Clearly, any admissible kernel satisfies  
\be
\aligned
& K(x,x) \geq 0, 
\qquad 
&&  K(x,y)^2 \leq K(x,x) \, K(y,y),  
\qquad x,y \in \Omega. 
\endaligned
\ee
The non-negative function 
$
D(x,y)= K(x,x) + K(y,y) - 2 K(x,y) \geq 0 
$
can be interpreted as a ``pseudo-distance'' in view of the properties $D(x,x) = 0$ and $D(x,y) = D(y,x)$. 
{  
The triangle inequality 
does not hold in general; yet this pseudo-distance provides us 
with a functional of interest as far as quantitative error estimates are concerned.
On the other hand, for some choices of kernel the triangle inequality does hold, as is the case with $K_1(x,y) =1- |x-y|$ on $[0,1]$ 
corresponding to the standard $\ell^1$ distance $D_1(x,y)= 2 \, |x-y|$.}


{ 

\paragraph*{Functional spaces.}

Given an admissible kernel $K:\Omega \times \Omega \to \RR$, we now introduce the (infinite dimensional) space $\Hcal^K(\Omega)$, referred to as (cf. for instance \cite{BFBL})
the {\it reproducing Hilbert space} generated from the kernel $K$ and 
 consisting of the {\sl completion}  of the set of 
all finite linear combinations of the functions $K(x, \cdot)$ parametrized by $x \in \Omega$, 
which we endow with the scalar product and norm defined as follows. 
To any two (finite, say) combinations 
$\varphi = \sum_{1 \leq m \leq N} a_m K(\cdot, y^m)$ and $\psi = \sum_{1 \leq n \leq N} b_n K(\cdot, y^n)$
for any $N$ distinct points $Y = (y^1,\ldots,y^N)$ in $\Omega$, 
we associate the bilinear expression (with $a = (a_m)$, etc.) 
\label{equa-bilin-norm}
\bel{Npsi}
\la \varphi, \psi \ra_{\Hcal^K(\Omega)} 
 = a^T K(Y,Y) b 
= 
 \sum_{1 \leq m \leq N} \sum_{1 \leq n \leq N} a_m b_n K(y^m,y^n), 
\ee
which endows the space with a Hilbertian structure with norm  
$\| \varphi \|_{\Hcal^K(\Omega)}^2  =   a^T K(Y,Y) a$. 
(Here, both $Y$ and $N$ may be arbitrary and thee limit $N \to +\infty$ is taken by completion.) 
By construction, the  reproducing kernel property holds in $\Hcal^K(\Omega)$: 
\bse
\label{equa-538}
\be
\la K(\cdot, x), K(\cdot, y) \ra_{\Hcal^K(\Omega)} = K(x,y), \qquad x, y \in \Omega. 
\ee
From the Cauchy-Schwarz inequality, it follows that for any $\varphi \in \Hcal^K(\Omega)$ and $x \in \Omega$ 
\bel{539} 
\aligned
|\varphi(x)| 
= | \la K(\cdot, x), \varphi \ra_{\Hcal^K(\Omega)}|  
& \leq \| K(\cdot, x) \|_{\Hcal^K(\Omega)} \, \|\varphi\|_{\Hcal^K(\Omega)}
 = \sqrt{K(x,x)} \, \|\varphi\|_{\Hcal^K(\Omega)}.
\endaligned
\ee   
\ese
Since the kernel is continuous and bounded, the ``point evaluation'' $\varphi \mapsto \varphi(x)$ is thus a linear and bounded functional on $\Hcal^K(\Omega)$ (for any $x \in \Omega$). 
} 
We also observe that 
$
|\varphi(x) - \varphi(y)| 
\leq
 D(x,y) \|\varphi\|_{\Hcal^K(\Omega)}$.
{
Since $K$ and thus $D$ are continuous in $\Omega$, all functions in $\Hcal^K(\Omega)$ 
are at least continuous.
}


\paragraph*{Discrepancy functional.}

{
Consider the integration error in \eqref{equation:MCE}, which we split into two parts, namely (1) 
a contribution estimating the {\sl regularity of the function} $\varphi$ under consideration, which is is independent of the choice of the interpolation points, and
(2)
a contribution depending solely upon the kernel $K$ and the mesh points, which is independent of the choice of the function. 
}
Various equivalent formulations of the second term can be derived \cite{PLF-JMM-estimate}, 
in physical variables, in spectral variables associated with the kernel, as well
as in terms of 
{
Fourier variables associated with a discrete lattice. }
Although these formulations are in principle equivalent, 
they shed a different light on the problem.   
For the factorization in physical variables, we find \eqref{equation-EE} as stated earlier. 


\subsection{Definition of discrete projection and differential operators}
\label{section-21}

\paragraph*{A scale of finite dimensional spaces.}

{

Concerning mesh-free strategies for the approximation of the gradient and Hessian operators, we refer to \cite{Liu-new1} for applications in structural dynamics as well as 
\cite{Liu-new2} for applications in fluid dynamics. The reader can also refer to
\cite{Afra-new,Nakano} for applications to heat conduction problems, while \cite{Chen-new} provides a state-of-the art review of the subject. We provide a here a systematic standpoint which encompasses a wide variety of discrete operators and aims at applications beyond the field of partial differential equations. 
}

Let us summarize our notation in the finite dimensional setup, in which we now fix an integer $N$, 
a finite collection of points $Y = (y^1, \ldots, y^N)$ in $\Omega$, together with an admissible kernel $K$ on 
$\Omega$. 
Based on these data, we define the finite dimensional vector space $\Hcal^K_Y(\Omega)$ consisting of all linear combinations of the so-called basis functions $x \mapsto K(x, y^n)$, that is, 
$\Hcal^K_Y(\Omega)= \big\{\sum_{1 \leq m \leq N} a_m K(\cdot, y^m) \,  / \,  a = (a_1, \ldots, a_N) \in \RR^N  \big\}. 
$
Since the kernel $K$ is continuous, $\Hcal^K_Y(\Omega)$ embeds into the space of all continuous functions on $\Omega$.  
We consider the bilinear expression \eqref{equa-bilin-norm} and the corresponding finite-dimensional 
Hilbert space $\Hcal^K_Y(\Omega)$   with norm  
$\| \varphi \|_{\Hcal^K_Y(\Omega)}^2 =   a^T K(Y,Y) a$. 
The reproducing kernel property (immediate from \eqref{Npsi}) reads 
$
\la K(\cdot, y^m), K(\cdot, y^n) \ra_{\Hcal^K_Y(\Omega)} 
= K(y^m, y^n)$,
and 
allows one to relate the coefficients $a_m$ of the decomposition of a function $\varphi = \sum_{1 \leq m \leq N} a_m K(\cdot, y^m)$ to its scalar product with the basis functions, namely
$ 
\varphi(y^n) 
=
\sum_{1 \leq m \leq N} a_m K(y^m, y^n) = a^T K(Y, y^n). 
$
{ This is nothing but the standard scalar product of the vectors $a$ and $K(Y, y^n)$ (for any given $n$). 
}


\paragraph*{$K$-projection.} 

{
Given any continuous function $f$ on $\Omega$, }
the vector $f_Y = \big( f(y^1), \ldots, f(y^N) \big)$ consists of the values of this function at the given points. We define $\Pbf^K_Y(f)$ in the space $\Hcal^K_Y(\Omega)$ by   reconstructing a suitable element of $\Hcal^K_Y(\Omega)$ from the given vector $f_Y$, specifically we set 
\begin{subequations}
\be
\Pbf^K_Y(f) (x)= a^T K(x, Y) = \sum_{1 \leq n \leq N} a_n K(x, y^n),  
\qquad x \in \Omega, 
\qquad a=  K(Y,Y)^{-1} f_Y. 
\ee
Clearly, $(\Pbf^K_Y \circ \Pbf^K_Y)(f) = \Pbf^K_Y(f)$ and, moreover, $\Pbf^K_Y(\varphi) = \varphi$ for any function $\varphi = a^T K(\cdot, Y) \in \Hcal^K_Y(\Omega)$. 
The norm of the $K$-projection operator $\Pbf^K_Y$ is 
\be
\| \Pbf^K_Y(f) \|_{\Hcal^K_Y(\Omega)}^2 =  f_Y^T K(Y,Y)^{-1} f_Y. 
\ee
\end{subequations}

A basis of functions is naturally associated with the discrete space $\Hcal^K_Y(\Omega)$,
and 
{
consists of the set of $N$ functions $\theta_Y^n: \Omega \to \RR$ taking the values $0$ or $1$ at the points of the set $Y$ such as each $\theta_Y^n$ interpolates a given vector in the canonical basis of $\RR^N$. 
}
 Precisely, using the notation $\theta_Y = (\theta^1_Y, \ldots, \theta^N_Y)$, we define the {\bf $K$-basis} (as a whole) by  
\bel{PU}
\theta_Y (x) = K(Y,Y)^{-1} K(Y, x), \qquad x \in \Omega. 
\ee
{ 
As expected,  it follows that, {for a family of distinct points and a positive kernel, we can write} 
}
\bse
\label{equa-215}
\be
\big( \theta^n_Y(y^m) \big)_{1 \leq n,m \leq N} = K(Y,Y)^{-1} K(Y,Y) = \Id
\ee
(the identity matrix) 
{ or, equivalently,
} $\theta_Y^n (y^m) = \delta_{nm}$ (using the Kronecker notation). 
On the other hand, the scalar product of any two basis functions is 
\bel{410}
\big\la \theta_Y^m,  \theta_Y^n \big \ra_{\Hcal^K_Y(\Omega)} 
=  K^{-1}(y^m, y^n),
\ee 
where $K^{-1}(y^m, y^n)$ stands for the $(n,m)$-coefficient of the matrix $K(Y,Y)^{-1}$. 
\ese
Importantly, this partition of unity is useful in expressing the {\bf $K$-discrete projection} of a function $f: \Omega \to \RR$, namely 
\bel{equa-216}
\Pbf^K_Y(f) = \sum_{1 \leq n \leq N} f(y^n) \theta_Y^n.
\ee
 
 
\paragraph*{$K$-gradient operator.}

{ 

\bse
Once more, we are given a set of points $Y=(y^1, \ldots, y^N)$ in $\Omega$ and we are interested in functions in 
the associated discrete space $\Hcal^K_Y(\Omega)$. 
Given an arbitrary (but at least continuous) function $h: \Omega \to \RR$, we define the $K$-gradient $\nabla^K_Y h$ as a suitable linear combination of the gradient of the basis functions with 
the weights $h_Y = h(Y) = (h(y^1), \ldots, h(y^N)) \in \RN$, namely, it is   
 the $D$-vector-valued function\footnote{ This expression is motivated from \eqref{PU} and \eqref{equa-215}.} 
\be
(\nabla^K_Y h)(x)  
= \sum_{n=1, \ldots, N} h(y^n) \nabla \theta^n   (x)
=   (\nabla K_Y)(x)  K(Y,Y)^{-1} h_Y \in \RR^D, 
\qquad x \in \Omega, 
\ee 
where $(\nabla K_Y)(x) = \big( \nabla K(y^1, x), \ldots, \nabla K(y^N, x) \big) \in \RR^{N\times D}$ is a rectangular matrix. 
This definition can be cast in the form of an operator (by suppressing the dependency in $h$), namely with 
\bel{grad}
\nabla K_Y =
\nabla K(Y,Y) = \big( \del_d K(y^n, y^m) \big)_{d = 1, \dots, D; n,m=1, \ldots, N} \in \RR^{D \times N\times N},
\ee
\ese
 the {\bf $K$-gradient operator} $\nabla^K_Y$ reads 
\bel{Grad}
\nabla^K_{Y} = \big( \del_d K(Y,Y) K(Y,Y)^{-1} \in \RR^{N \times N}\big)_{d=1, \ldots, D} = \nabla K_Y K_Y^{-1}     \in \RR^{D \times N \times N}.
\ee

}


\paragraph*{Other $K$-operators.}

We can also define many other integro-differential operators of interest in the applications. 
For instance, we can introduce the {\bf $K$-divergence operator} $(\nabla^K_Y)^T$, where the transposition is defined with respect to 
the standard Euclidian product, i.e.~for any vectors $h_Y \in \RR^N$ and $f_Y \in \RR^{N\times D}$ we write 
\bel{equa-diverg} 
<h_Y, (\nabla_Y)^T f_Y>_{\RR^{N }}  
= 
<\nabla^K_Y h_Y, f_Y>_{\RR^{N \times D}}. 
\ee
{
Observe that this notion 
depends upon the choice of the kernel. 
}

{
Furthermore, 
the {\bf $K$-Laplacian operator} is defined as $\Delta^K_Y = \nabla^K_Y \cdot \nabla^K_Y \in \RR^{N \times N}$ and is
 the discrete counterpart of the standard Laplacian operator $\Delta$. 
 }
{ We apply the $K$-gradient operator $\nabla^K_Y \in \RR^{D \times N \times N}$ twice and sum over in the index $d=1, \ldots, D$ and, more explicitly, 
\be
\Delta^K_Y
= \sum_{d=1, \ldots, D}
\big( \del_d K(Y,Y) K(Y,Y)^{-1} 
\big) 
 \big( \del_d K(Y,Y) K(Y,Y)^{-1}  
\big) 
\ee
It involves 
a contraction with respect to the indices $d=1, \ldots, D$ and matrix multiplication in the indices $n=1, \ldots, N$. 
}
{ Along exactly the same lines, the {\bf $K$-Hessian operator} is defined as $\nabla_Y^{K, 2} = \nabla^K_Y \otimes \nabla^K_Y 
 \in \RR^{D \times D \times N \times N}$
 }
 and 
  is the discrete counterpart of the standard Hessian operator $\nabla^2$ and 
involves a standard matrix multiplication in the two indices ranging $1, \ldots, N$.  

\bse
{ 
Finally, $\mathcal{L}_Y$ (which is an approximation to the Leray operator, i.e.~the projection on divergence-free vector fields)
and $\Pibf^K_Y$ are introduced by setting, 
for any discrete vector field $S_Y \in \RR^{N \times D}$,  
}
\bel{hodge1}
   \Pibf^K_Y S_Y = \big(\nabla^K_Y \big((\nabla^K_Y)^T \nabla^K_Y \big)^{-1} (\nabla^K_Y)^T S_Y \big), 
   \qquad \zeta_Y = \mathcal{L}_Y S_Y = \big( \Id - \Pibf^K_Y \big) S_Y, 
\ee
and we arrive at the orthogonal {\bf $K$-Hodge decomposition} 
\be
\label{hodgeY}
S_Y = \nabla^K_Y (h_Y + f_Y) + \zeta_Y,
\qquad
 (\nabla^K_Y)^T  \zeta_Y  = 0, 
\qquad
(\nabla^K_Y)^T \nabla^K_Y f_Y= 0, 
\ee
where $f_Y$ represents a discrete harmonic component. 
 \ese
{
This decomposition is modeled upon the standard Hodge decomposition, characterized by 
%
$S = \nabla (h + h_0) + \zeta$
and
$
 \nabla \cdot \zeta \mu  = 0$, 
together with 
$
 \zeta \cdot \eta = 0$ and 
 $\nabla \cdot(\nabla h \mu) = 0$,  
%
where $h_0$ has compact support and $\mu \simeq \frac{1}{N} \sum_{n=1}^N \delta_{y^n}$.
 }


\subsection{Definition of discrete interpolation and extrapolation operators}
\label{sect--244}

\paragraph*{Notation.}

 {
An admissible kernel $K$ is fixed throughout. 
Consider an arbitrary collection of $M$ points $Y$ in $\Omega \subset \RD$ and a vector-valued function $f  =f(y) \in \RR^P$, which provide us with the two sets of data 
$Y \in \RR^{M \times D}$ and $f_Y \in \RR^{M \times P}$. 
Consider also a collection $X$ of $N$ points in $\Omega \subset \RR^D$ with $N << M$ 
{ 
and let us associate with these two sets
the corresponding functional spaces $\Hcal^K_Y$ and $\Hcal^K_X$.
}
Finally we introduce the rectangular matrix 
{
consisting of all relevant values of the kernel 
$
K(X,Y) = \big( K(x^n,y^m) \big)_{1 \leq n \leq N\atop 1 \leq m \leq M} \in \RR^{N \times M}$, 
together with its discrete gradient 
   $\nabla K(X,Y)= \big( \nabla_{y^m} K(x^n,y^m) \big)
   \in \RR^{N \times D \times M}.
$
}}
 

\paragraph*{$K$-extrapolation.}

Suppose that we are only given some data $f_X \in  \RR^{N \times T}$ on a small set of points $X$. Then we can {
extrapolate these values to a larger set of points}
 $Y$ by introducing\footnote{ Here, the term "extrapolation" is used when some data defined on a ``small' set are extended to a "large" set. Statistician reader might prefer a different terminology here.} 
\bel{equation:EX-00}
\Extra_{X,Y}^K f_X = K(Y,X) K(X,X)^{-1} f_X \in \RR^{M \times T}.
\ee
Similarly, the $K$-extrapolation of the gradient operator on $Y$ is defined as 
\bel{equation:NEX}
\nabla \Extra_{X,Y}^K f_X = (\nabla K)(Y,X) K(X,X)^{-1} f_X \in \RR^{M \times D \times T}. 
\ee


\paragraph*{$K$-interpolation.}

\bse
Conversely, suppose that we are given some data $f_X \in  \RR^{N \times P}$ on a set $X \in \RR^{N \times D}$. Then for any collection $Y \in  \RR^{M \times D}$, we define $\Inter_{X,Y}^K f_X \in  \RR^{M \times P}$ 
{
by the solution of the minimization problem 
}
\be
\Inter_{X,Y}^K f_X = \arg \inf_{g_Y \in \RR^{M \times P}} \| K(Y,X)K(Y,Y)^{-1} g_Y - f_X\|,
\ee
\ese
which yields us 
\bel{IN}
f_Y = \Inter_{X,Y}^K f_X = K(Y,Y)K(X,Y)^{-1} f_X \in \RR^{M \times P}. 
\ee
Here, by definition, $K(X,Y)^{-1}$ is the standard (least-square) inverse matrix, characterized by the conditions
$$
\aligned
K(X,Y)^{-1} &= \big( K(Y,X) K(X,Y)\big)^{-1} K(Y,X) \in \RR^{M \times N}, \qquad M \leq N,
\\
K(X,Y)^{-1} &= K(Y,X) \big( K(X,Y) K(Y,X)\big)^{-1} \in \RR^{M \times N}, \qquad N \leq M.
\endaligned
$$
In particular, we have the identity $\Extra_{Y,X}^K f_Y = f_X$.


\paragraph*{$K$-projections.}

Interpolation and extrapolation are particular cases of the projection operators that we define now. 
Let $(X,Y,Z) \in (\RR^{M \times D},\RR^{N \times D},\RR^{P \times D})$ be collections of $D$-vectors, 
and let us assume the restriction $M \leq N$. 
Then we define $\Pibf_{X,Y,Z}$  by 
\bel{equation:PO}
\Pibf_{X,Y,Z}^K = \Extra_{Y,Z}^K \Inter_{X,Y}^K f_Y = K(Z,Y) K(X,Y)^{-1}. 
\ee
For instance, $\Extra_{X,Y}^K = \Pibf^K_{X,X,Y}$ and $\Inter_{X,Y}^K = \Pibf^K_{X,Y,Y}$. A 
particularly interesting operator is $\Pibf^K_{X,Y,X}$, whose kernel ``contains the information''
 lost from approximating a function from a larger set $X$ to a smaller set $Y$.


\subsection{A steepest descent algorithm}

\paragraph*{$K$-discrepancy error.}

Consider two sets of points $X,Y$ as above. In agreement with \eqref{equation-EE}, we compute the discrepancy error between the discrete measures $\mu_Y$ and $ \mu_X$ as
\bel{EEPO}
\dK(\mu_Y, \mu_X)^2
 = \frac{1}{M^2}  \sum_{l,m=1}^M K(y^l,y^m) + {1 \over N^2} \sum_{n,q=1}^N K(x^n,x^q)
- \frac{2}{N M} \sum_{n=1}^N \sum_{m=1}^M  K(x^n,y^m).
\ee
The notion of discrepancy error $\dK(\mu_Y, \mu_X)^2$ also allows us to define {\sl transportation maps,}
{ 
and a path of points $t \mapsto (X_t)$, starting from $X_0 = X$ and reaching   
}
$X_\infty = \arg_{X \in \RR^{N \times D}} \inf \dK(\mu_Y, \mu_X)^2$.
Precisely, for each $t$ we explicitly compute the gradient of $\dK(\mu_Y, \mu_{X_t})^2$ with respect to $X_t$ as
\be
\nabla_{x^n} \dK(\mu_Y, \mu_{X_t})^2
 = \frac{2}{N^2} \sum_{q=1}^N \nabla_{x^n} K(x_t^n,x_t^q)  
- \frac{2}{N M} \sum_{m = 1}^M \nabla_{x^n} K(x_t^n,y^m).
\ee
In view of this, we introduce the following semi-discrete scheme corresponding to a steepest descent algorithm: 
\bel{SDA}
\frac{d}{dt} x^n_t =  \frac{2}{N M} \sum_{m = 1}^M \nabla_{x^n_t} K(x^n_t,y^m) - \frac{2}{N^2}\sum_{q=1}^N  \nabla_{x^n_t} K(x^n_t,x^q_t), 
\ee
with the initial data $X_0 = X$. Summarizing our notation, the scheme can be written as
\be 
\frac{d}{dt} X_t =  \frac{2}{N M} \nabla_{X_t} K(X_t,Y)_{3} - \frac{1}{N^2} \nabla_{X_t} K(X_t,X_t)_{3}, 
\ee
{ 
where the notation $_{3}$ denotes the trace in the third index, namely 
$\nabla_{X_t} K(X_t,Y)_{3}$ denotes the set of points $\sum_{m = 1}^M \nabla_{x^n_t} K(x^n_t,y^m)$.
}
The trajectories $t \mapsto X_t \in \RR^{N \times D}$ follow characteristics corresponding to optimal transportation maps.
 
This algorithm also applies to probability distributions $x \mapsto \mu(x)$ that need not be discrete, and 
using the discrepancy error function \eqref{equation-EE}, we can compute
\bel{EEPO2}
\dK(\mu, \mu_{X_t})^2
 = \int \int K(x,y) \, d\mu(x)d\mu(y) + {1 \over N^2} \sum_{n,m}^N K(x_t^n,x_t^m)
- \frac{2}{N} \int \sum_{n=1}^N K(x_t^n, x) \, d\mu(x),
\ee
so that 
\bel{EEPO3} 
\frac{d}{dt} \dK(\mu, \mu_{X_t})^2 =  \langle \frac{2}{N^2} \nabla_{X_t} K(X_t,X_t)_{3} - \frac{2}{N} \int \nabla_{X_t} K(x,X_t) d\mu(x), \frac{d}{dt} X_t \rangle. 
\ee
The proposed scheme can be written as 
\bel{EEPO4}
\frac{d}{dt} X_t =  \frac{2}{N} \int \nabla_{X_t} K(x,X_t) d\mu(x) - \frac{2}{N^2} \nabla_{X_t} K(X_t,X_t)_{3}
\ee
and, therefore, requires the computation of the integral in the right-hand side.


\paragraph*{Fitting functions.}

\bse
Given some data $Y \in \RR^{M \times D}$ and $f_Y \in \RR^{M \times P}$,
{ 
we can consider the projection operator determined from \eqref{equation:PO}, that is,  
$\Pibf^K_{Y,X,Y}$ denoted for short by $\Pibf^K_{X,Y}$.
}
Starting from $X_0 = X \in \RR^{N \times D}$, fitting a function corresponds to determine a distribution of points $X_\infty$ achieving the following minimum: 
\be
\| f_Y - \Pibf^K_{X_\infty,Y}f_Y\|^2 = \inf_{X \in \RR^{N \times D}} \| f_Y - \Pibf^K_{X,Y}f_Y\|^2. 
\ee
We rewrite this minimization problem as a constrained problem, 
{ 
namely (with $Y$ and $f_Y$ fixed): }
\be
\| f_Y - \Pibf^K_{X_\infty,Y}f_Y\|^2 = \inf_{X \in \RR^{N \times D}} \| f_Y - \Extra_{Y,X}^K f_X\|^2, 
\qquad f_X = I(X,Y)f_Y, 
\ee
\ese
where $ \Extra^K$ denotes the $K$-extrapolation operator. During the minimization step, we consider that $f_X \in \RR^{N \times P}$ is independent of $X$ and we take the constraint into account and we compute
\be
\nabla_{X} \| f_Y - \Extra_{Y,X}^K f_X\|^2 = \langle \Extra_{Y,X}^K f_X - f_Y, \nabla_{X} \Extra_{Y,X}^K f_X \rangle \in \RR^{N \times D}, 
\ee
{ 
where the right-hand side is the multiplication  between $\Extra_{Y,X}^K f_X - f_Y \in \RR^{M \times P}$ and the discrete operator 
$
\nabla_{X} \Extra^K_{Y,X} f_X \in \RR^{M \times N \times D \times P}.
$
}


\paragraph*{Steepest descent algorithm.}

\bse

{Based on the error expression 
$
\ebf(X,Y) = \Extra_{Y,X}^K f_X - f_Y \in \RR^{M\times P}, 
$
we define a family of points $(X_t)$ (parametrized by $t$) 
by solving the steepest descent algorithm
\be
\frac{d}{dt} X_t = -\langle \ebf(X_t,Y), (\nabla_{X_t} \Extra^K)_{Y,X_t)} f_{X_t} \rangle_{\RR^{M \times P}}.
\ee
Using \eqref{equation:NEX}, we compute 
$\nabla_{X} \Extra_{Y,X}^K
$
and  our scheme reads 
\be
\aligned
\frac{d}{dt} X_t 
& = - \langle \ebf(Y,X_t), (\nabla_{X_t} K)(Y,X_t) K(X_t,X_t)^{-1} f(X_t) \rangle_{\RR^{M \times P}} 
\\
& \quad +\langle K(X_t,X_t)^{-1} K (X_t,Y) e(X_t,Y), (\nabla_{X_t} K)(X_t,X_t) K(X_t,X_t)^{-1} f(X_t) \rangle_{\RR^{N \times P}}, 
\endaligned
\ee
}
and we also observe that 
$
K (X,Y) \ebf(X,Y) 
= 0. 
$
We can write our scheme component-wise as 
\be
\frac{d}{dt} x_t^n = -\sum_{m,p = 1}^{M,P} \ebf(X_t,Y)_{m,p}  (\nabla_{x_t^n} K)(y^m,x_t^n) cf(X)_{n,p}. 
\ee

\ese



\section{An algorithm for support vector machines in machine learning} 
\label{sec-trois}

\subsection{Error estimates for support vector machine}
\label{sec-31}

\paragraph*{Standpoint in this section.}

Our aim in this section is to formulate a support vector machine in the context of the kernel theory in Section 2, and
 emphasize the role of error estimates based on sharp discrepancy sequences. Interestingly, our formulation uses an intermediate space (of size denoted by $N_y$, below) which avoids the computation of the inverse of the kernel
 and therefore leads us to an algorithm that is linear with respect to, both, the input and the output variables. 
In the standard formulation, there is no such pivot space, while we observe here that this notion simplifies the computational complexity. 
 
We thus consider neural networks based on the concept of support-vector machine (SVM). Our presentation in this section\footnote{ 
Our notation may be unconventional for a statistician reader, as we use $Y$ for a
 regression variable (and not for a dependent variable).} 
 is based on the notion of discrete projections, interpolation, and extrapolation operators introduced in the previous section. 
Consequently, our algorithms come with quantitative error estimates of the type we investigated in \cite{PLF-JMM-estimate}.


\paragraph*{Setup of interest.}

An admissible kernel $K : \RR^D \times \RR^D \to \RR$ being fixed once for all, we work with several 
sets of points, denoted by $X \in \RR^{ N_x \times D}$, $Y \in \RR^{ N_y \times D}$, and $Z \in \RR^{N_z \times D}$. {
 Broadly said, our aim is to determine the values $f(Z)$ of an unknown function $f$ on a ``large'' set of points $Z$ 
from the knowledge of its values $f(X)$ on a ``small'' set of points. The role of the set of points $Y$ is secondary, but very useful in practice (see paragraph above on the pivot space). 
}
Using the kernel, we compute the $(N_x \times N_y)$ Gram matrix $ K(X,Y) = \big( K(x^n,y^m)\big)_{1 \leq n\leq N_x \atop 1 \leq n\leq N_y}$ associated with the sets $X$ and $Y$.
Assuming that $N_y \le N_x$, a \textbf{feed-forward SVM}, by definition, is the $K$-projection operator 
associated with the set of points $X,Y,Z$ defined by the matrix
\bel{equa-projection}
P(X,Y,Z) =  K(Z,Y) \, K(X,Y)^{-1} \in \RR^{N_z \times N_x}, 
\ee
where the inverse of a (rectangular) matrix $A$ is defined as usual by $A^{-1} = (A^TA)^{-1}A^T$. 

To any vector-valued function $\varphi: \RD \to \RR^P$ and any set of $D$ points $X$, we associate the matrix $\varphi(X) \in \RR^{N_x \times P}$, and we introduce its $K$-projection or \textbf{prediction} as 
\be
\varphi_Z =  P(X,Y,Z) \varphi(X) \in \RR^{N_z \times P} 
\ee
and we also have its discrete gradient  
$
\nabla^K \varphi_Z = (\nabla_z K)(Z,Y) \, K(X,Y)^{-1} \varphi(X) \in \RR^{D \times N_z \times P}. 
$
{
As explained in the previous section, 
}
from $\nabla^K$ we can also compute other differential operators such as
 $\Delta^K$ and 
 $\nabla^{K,2}$.


\paragraph*{Error estimate.}
  
Given a SVM denoted by $\mathcal{P}(x,y,z)$, the { error estimate available} the expectation of a function $\varphi$ reads 
\bel{equa344} 
\Big| \frac{1}{N_z}\sum_{n=1}^{N_z} \varphi(z^n) - \frac{1}{N_z}\sum_{n=1}^{N_z} \varphi_{z^n} \Big| \le \Big( \dK\big(\mu_X,\mu_Y\big) + \dK\big(\mu_Y,\mu_Z\big) \Big) \|\varphi\|_{\Hcal^K(\Rd)}, 
\ee
in which the discrepancy error is now given by the discrete distance 
\be
\dK\big(\mu_X,\mu_Y\big) 
= \frac{1}{N_x^2}\sum_{n=1,m=1}^{N_x,N_x} K\big(x^n,x^m\big) 
+ \frac{1}{N_y^2}\sum_{n=1,m=1}^{N_y,N_y} K\big(y^n,y^m\big) 
- \frac{2}{N_x N_y}\sum_{n=1,m=1}^{N_x,N_y} K\big(x^n,y^m\big). 
\ee
The bound in \eqref{equa344}  involves now {\sl two} contributions, namely the relative distances between the three distributions $\mu_X, \mu_Y, \mu_Z$, with the notation 
$\mu_X = {1 \over N} \sum_{=1}^{N_x} \mu_{x^n}$. 
On the other hand, the norm arising in \eqref{equa344} can be approximated by 
$\|\varphi\|_{\Hcal^K }^2 \simeq \varphi_Y^T K(Y,Y)^{-1} \varphi_Y$.


\paragraph*{Neural networks as kernel methods.}

{
Our setup encompasses the classical strategy adopted in the theory of deep learning, which is based on the notion of neural networks. 
Let us consider the feed-forward neural network of depth $M$ 
 defined by induction ($m=0, \ldots, M$) by the relations 
\bel{equa-all-z} 
z_m = y_{m} g_{m-1}(z_{m-1})\in \RR^{N_{m}}, 
\qquad z_0 = y_0 \in \RR^{N_0}, 
\ee
in which $N_0, N_1, \ldots, N_M$ are integers and 
$g_{m}:  \RR^{N_{m-1}} \to  \RR^{N_{m}}$ are prescribed {\sl activation functions}, while 
$y_0,\ldots,y_M$ are prescribed {\sl weights} with $y_m \in \RR^{N_{m} \times N_{m-1}}$. 
{
By direct substitution, we arrive at a function of the variable $y=(y_0, \ldots,y_{M-1}, y_M)$, namely 
$
z_M(y) = y_M z_{M-1}(y_0, \ldots,y_{M-1}) : \RR^{N_0 \times \ldots \times N_M} \to \RR^{N_M }$, 
in which the function $z_{M-1}$ is defined inductively from $z_{M-2}$, etc., as described in \eqref{equa-all-z}. 
With this notation, using the family of kernels $k_m(x,x') = g_{m-1}(x x'{}^T)$ 
 the neural network of interest is thus entirely represented by the kernel compositions ($m=0, \ldots, M$) 
{ 
\be
K(y_m,\ldots,y_0) = K_m \Big(y_m, K_{m-1} \big( \ldots, K_1(y_1,y_0) \big) \Big) \in \RR^{ N_m \times \ldots \times N_0}. 
\ee 
}
}
%


\subsection{Numerical results}
 
\paragraph*{Dimension one.}
 
We present a rather basic test, while referring to \cite{PLF-JMM-SM-Tutorial} for further numerical examples.
We consider the role of kernels and we illustrate graphically the extrapolation operator \eqref{equation:EX-00} as well as the discrepancy error and functional norm computed from the kernels. 
We recall our notation: $x \in \RR^{N_x\times D}$ and $f(x) \in \RR^{N_x \times D_f}$ for
 the training set and the function values, respectively. 
 {
 On the other hand, $z \in \RR^{N_z \times D}$ and $f(z) \in \RR^{N_z \times D_f}$ represent
 the test set and the reference values respectively. 
Our test is based on the following data: a function $f$, a kernel $k$, and five integers $D,N_x,N_y,N_z,D_f$, together with 
the choice of function  
}
\bel{equa--2D}
f(x) = \prod_{d=1, \ldots, D} \cos (4\pi x_d) + \sum_{d=1, \ldots, D} x_d.
\ee
We use a random generator, configured to select points $x \in [-1,1]^{N_x \times D}$ randomly in the unit cube, and 
we pick up 
$z \in [-1.5,1.5]^{N_x \times D}$ as uniformly distributed over another cube, to observe extrapolation effects. The values of $x,z$ as well as corresponding  function values are shown in Figure~\ref{figure---31}.

\begin{figure}
\centering
\includegraphics[width=.6\linewidth, height=0.30\linewidth]{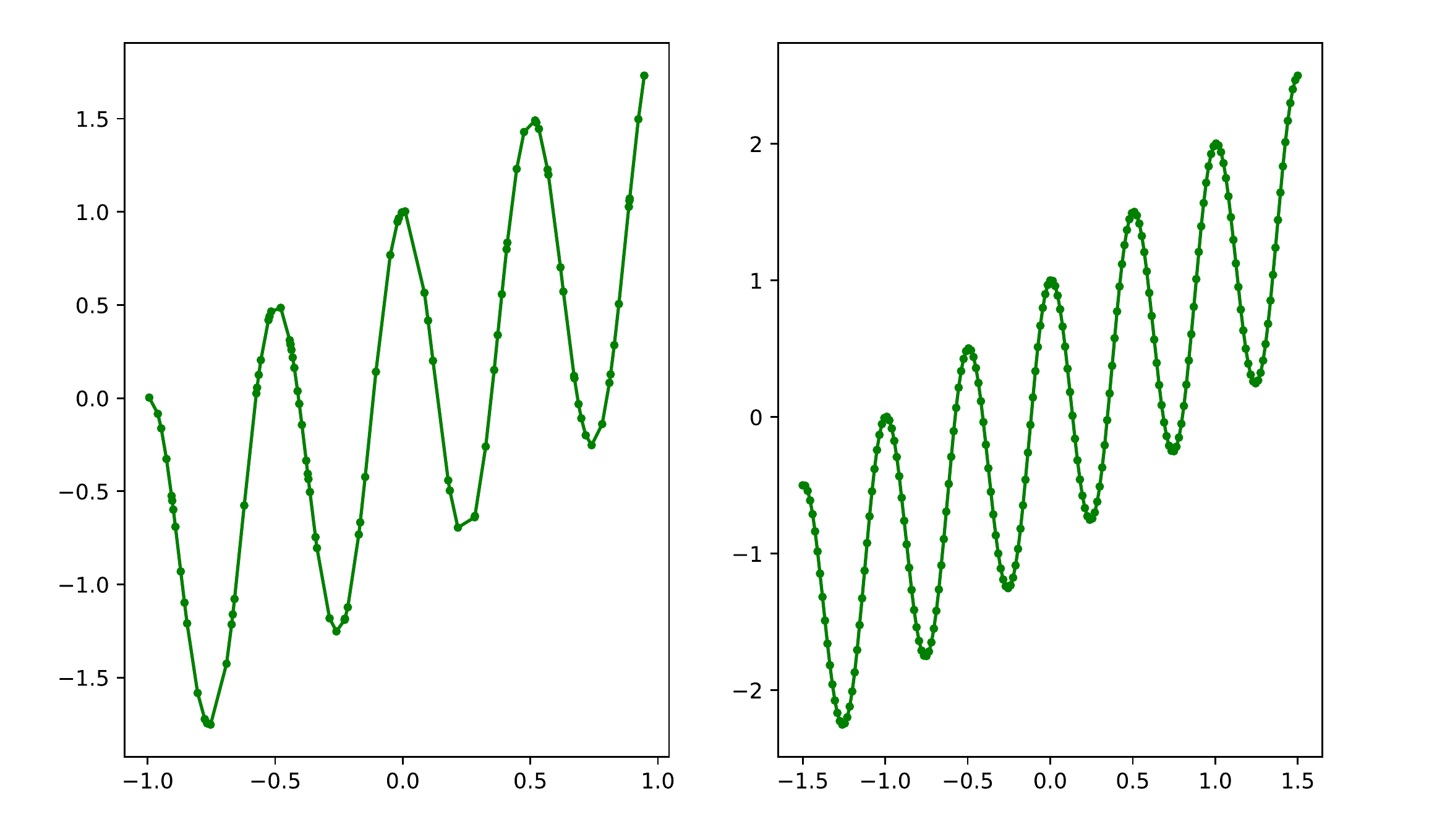} 
\centerline{Axis: (a) horizontal $x$,  vertical $f$; \qquad (b)  horizontal  $x$, vertical $f$}
\caption{  
(a): training set and function $(x,f(x))$. 
(b): test set and function $(x,f(x))$. 
} 
\label{figure---31}
\end{figure}

\begin{figure}
\centering
\includegraphics[width=.9\linewidth, height=0.40\linewidth]{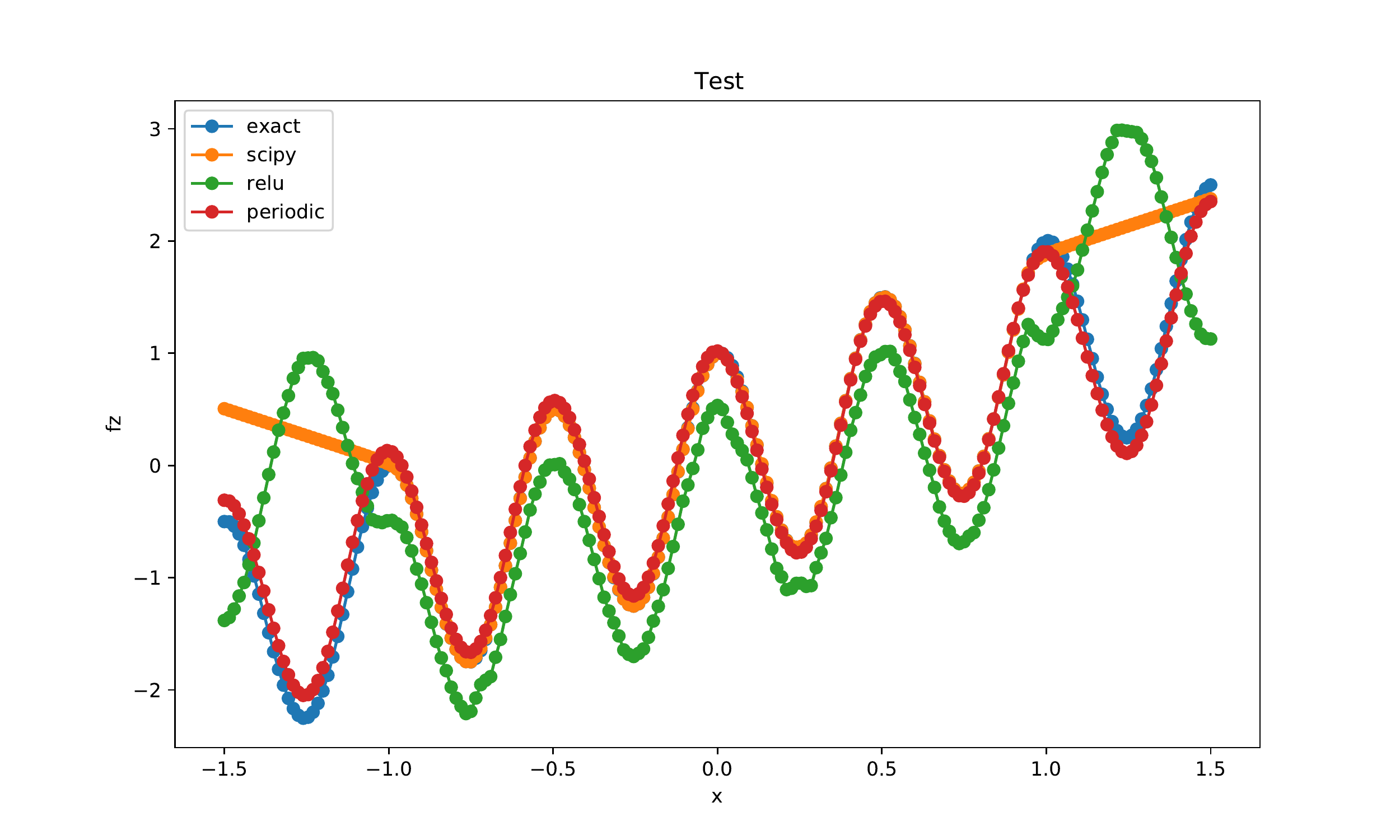}
\centerline{Axis: horizontal  $x$, vertical $f$}
\caption{  Extrapolation operator \eqref{equation:PO} using three different kernels.
}
\label{figure---31-show}
\end{figure}


We then test the extrapolation operator \eqref{equation:PO} using three different kernels. The first one is a Gaussian kernel, named SciPy, as we compared its results with the SciPy implementation to check our results. The second one is a kernel equivalent to a RELU network. Finally the third one is a Gaussian and periodic kernel. This kernel obviously gives more accurate results for this test, as shown in Figure~\ref{figure---31-show}. 
Observe that for Scipy and RELU, we extrapolate the data $(x,f(x))$ from the interval $[-1,1]$ to the larger interval $[-1.5,1.5]$.
In addition, for RELU  we also provide an illustration that a neural network method need not enjoy the reproducing property: if a point
$(x,(f(x))$ belongs to the training set and the extrapolation $z \to P(z)$ by a neural network may well have $f(x) \neq P(x)$.
Finally, in the following table we show the discrepancy error, as well as the kernel-norm of the computed functions: 
\begin{longtable}[]{@{}l|lll|@{}}
\toprule
& Scipy & RELU & periodic\tabularnewline
\hline
disc. error & 0.336247 & 0.2725986 & 0.1282466\tabularnewline
function norm & 0.7779849 & 0.7779849 & 0.7779849\tabularnewline
\bottomrule 
\end{longtable}

  
%
  
\paragraph*{Dimension two.}
 
The previous test can be performed in any dimension and we show here the two-dimensional case ($D=2$) 
{ with
 the same choice of function \eqref{equa--2D}.}
 The plot of  $x,z,f(x),f(z)$ is given in Figure~\ref{fig-4930}. 
Next, in Figure~\ref{fig-43DF0} we illustrate the effect of extrapolating data outside the training set by using a Gaussian kernel. 
On the other hand, in Figure~\ref{fig-DFEF0} we show that the extrapolation is better dealt with a periodic kernel. 
Finally, we provide an error analysis as follow:
\begin{longtable}[]{@{}l|ll|l|@{}}
\toprule
& Gaussian & periodic\tabularnewline
\hline
disc. error & 0.2105513 & 0.05348524\tabularnewline
function norm & 8181.435 & 2.000\tabularnewline
\bottomrule
\end{longtable}

\newpage 

\begin{figure}
\centering
\includegraphics[width=.9\linewidth, height=0.30\linewidth]{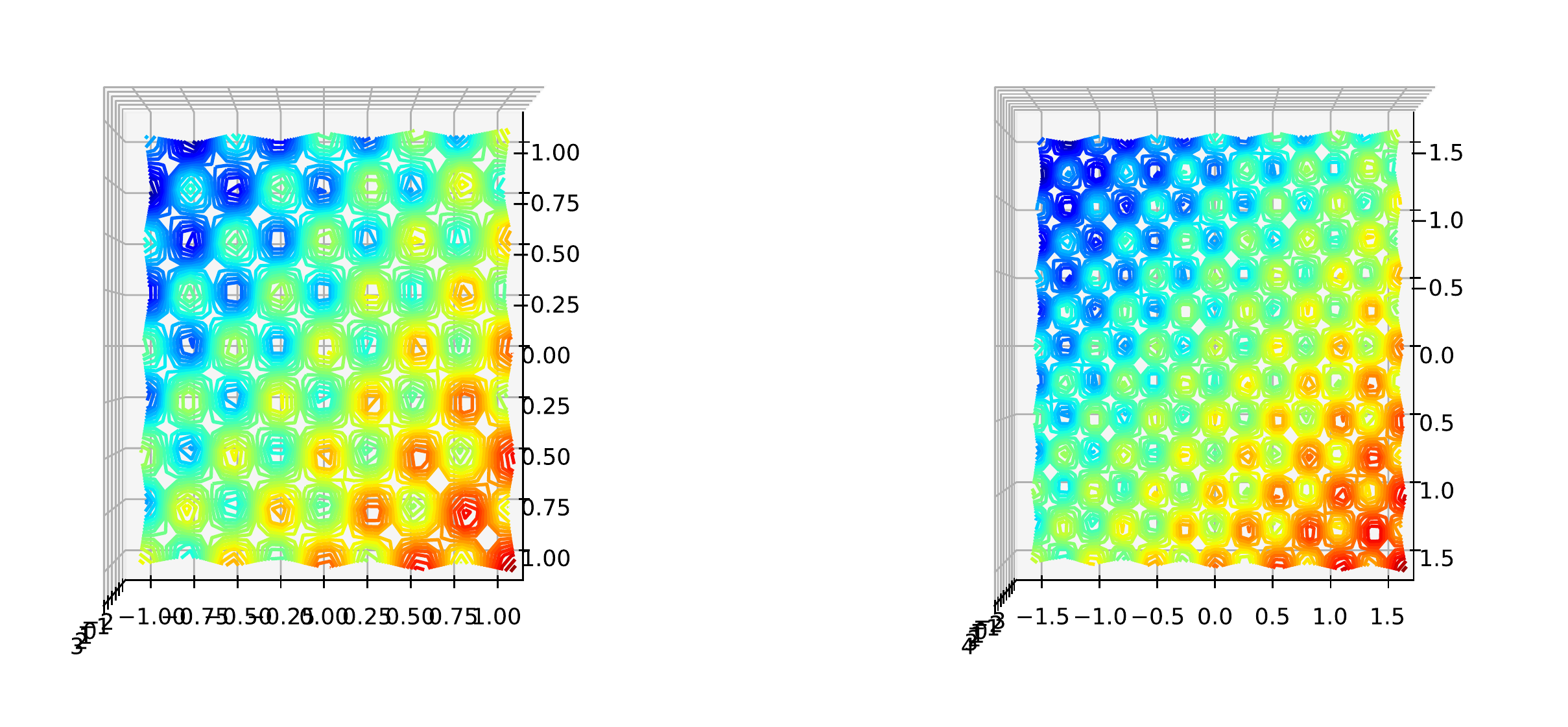} 
\centerline{Axis: horizontal  $x_1$, perspective $x_2$, vertical $f$}
\caption{  Plot of $x,z,f(x),f(z)$ in dimension $D=2$.
}
\label{fig-4930}
\end{figure}

\begin{figure}
\centering
\includegraphics[width=.9\linewidth, height=0.30\linewidth]{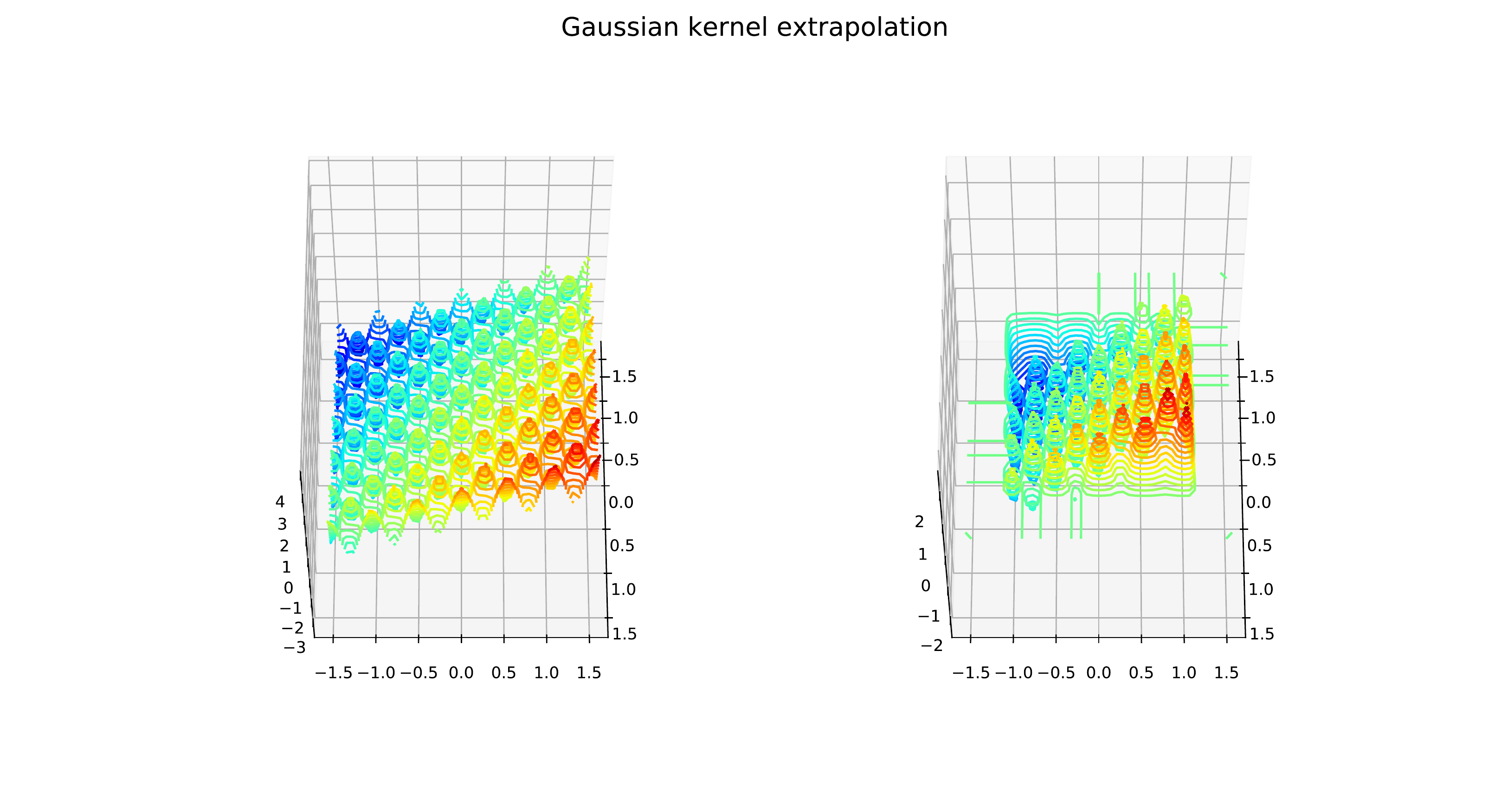} 
\centerline{Axis: horizontal  $x_1$, perspective $x_2$, vertical $f$}
\caption{  Extrapolating data outside the training set by using a Gaussian kernel.}
\label{fig-43DF0}
\end{figure}

\begin{figure}
\centering
\includegraphics[width=.9\linewidth, height=0.30\linewidth]{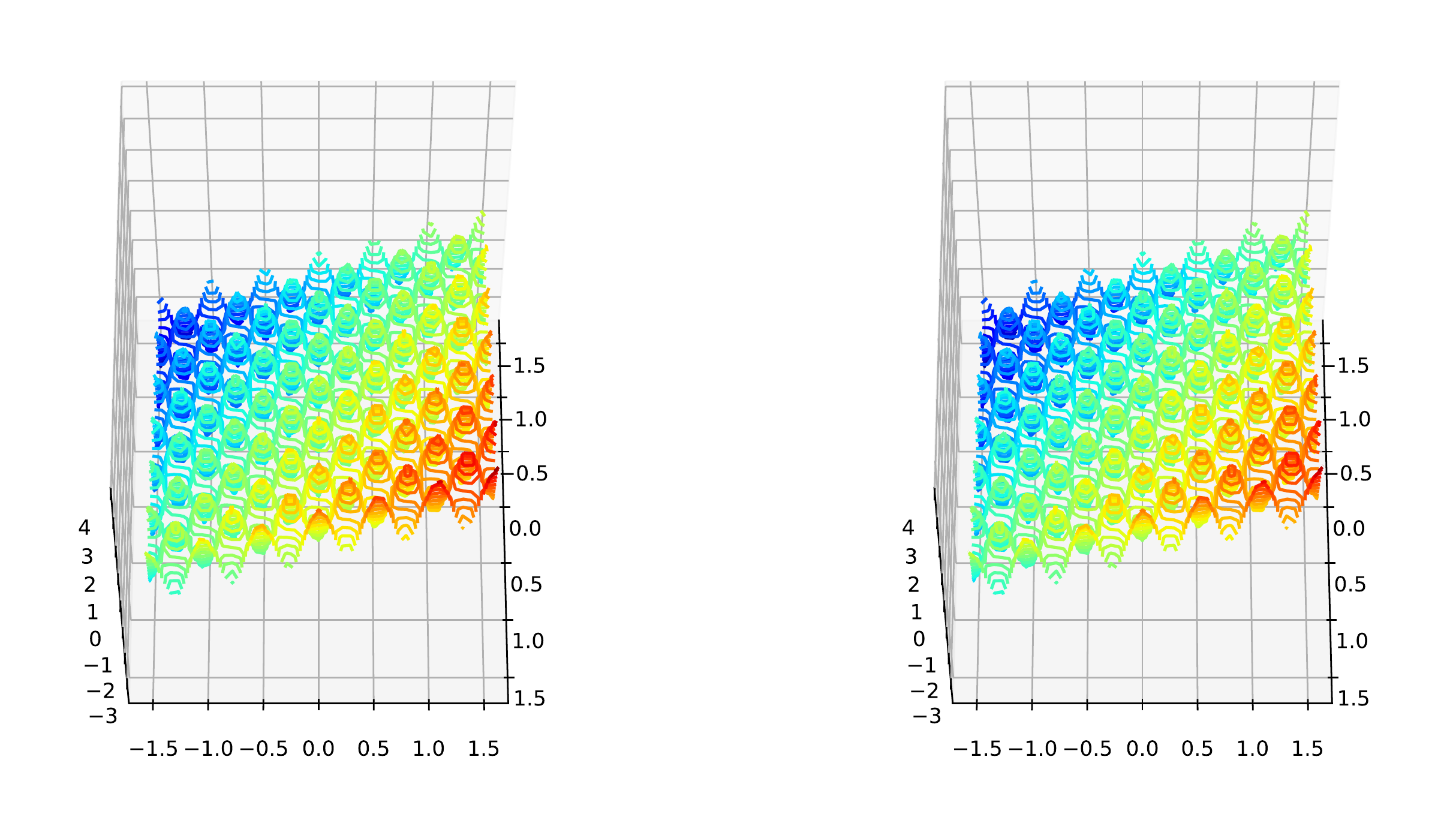} 
\centerline{Axis: horizontal  $x_1$, perspective $x_2$, vertical $f$}
\caption{  Extrapolating data outside the training set by using a periodic kernel.}
\label{fig-DFEF0} 
\end{figure}

\clearpage


\subsection{Pattern recognition problem with MNIST}

\paragraph*{Algorithm applied to the MNIST database.} 

We now use our algorithm for solving a very standard problem in machine learning. The MNIST database 
 \cite{LCB} which provides one with a real-word test for comparing numerical algorithms of pattern recognition. 
 We build upon the discussion in Section~\ref{sec-31} and consider a typical setup. 

\bei 

\item We are given a {\sl training set of $M =60000$ images,} each representing $T=10$ handwritten digits.
To each image we attach a label which is an integer between $0$ and $9$. Each image is a matrix of $28 \times 28$ scalars representing gray scale, which we display as a $D$-vector in dimension $D=784$. 

\item We also consider a {\sl prediction set} consisting of $P=10000$ images, which we will use to compute the error of our prediction. 

\item Moreover, we introduce a {\sl regression set} consisting of {\sl $N$ randomly chosen images} (with  $N \leq M$) among the $M$ images of our training set. 
\eei 
 
\noindent We use the notation $N=N_Y$, $P=N_Z$, and $M=N_X$. 
and we summarize our notation in the following table: 
\begin{center}
\begin{tabular}{|c|c|c|c|c|c|c|c|}
\hline 
$X \in \RR^{M \times D}$ & $Y \in \RR^{N \times D}$ & $Z \in \RR^{P \times D}$ & $L_X \in \RR^{M}$ & $L_Y \in \RR^{N}$ & $L_Z \in \RR^{P}$ \\ 
\hline 
training images & regression  & prediction  & training labels & regression & prediction 
\\ 
\hline 
\end{tabular} 
\end{center}


Throughout, an admissible kernel denoted by $K$ is fixed. To any discrete vector-valued function $f_X \in \RR^{M\times T}$, as explained earlier we associate the {\sl extrapolation function} 
$
f_X(Z) = K(Z,X) K(X,X)^{-1} f_X \in \RR^T, 
\qquad Z \in \RD. 
$
Once the training of a given model $L$ is performed, we predict its value, denoted $L(z)$ for an image $z$. The accuracy of the numerical output is evaluated by computing 
\bel{equa-erro}
E(N,M,L) = {1 \over P} \#_{z \in Z} \big\{z \, / \, L_Z(z) = L(z) \big\}. 
\ee

\paragraph*{An illustrative test.} 

Figure~\ref{keras} show the results based on a neural network of $N$ dense layers, trained with $M$ images, 
with Adam optimizer and sparse categorical cross-entropy loss function\footnote{as define at https://www.tensorflow.org/tutorials/quickstart/beginner }. We will use this test as a comparison for our results, since 
we can implement it on the same computer and is CPU-parallel, as are all of our tests.
Our numerical tests are based on a quite standard kernel, that is,
 the transported Gaussian kernel 
 \be 
 K(x,y) = \exp\Big(-C\, |\text{erf}^{-1}(2x-1)-\text{erf}^{-1}(2y-1)|^2 \Big),
 \ee 
 where 
 $\text{erf}^{-1}$ denote the inverse of the standard error function and $C$ is a scale factor. 
 Using transported kernels is natural in the present setting, since these kernels are compactly supported as are our data,
 after a rescaling to the set $[0,1]^{784}$.
In fact, we present our results with the kernel giving the best results among all other kernels we have tested 
{
(Mat\'ern, tensorial Mat\'ern, multiquadrics, RELU, linear regression). 
}
However, all numerical results are quite comparable, except for linear regression kernels which exhibit much lower accuracy.

 \begin{figure}
\centering
\begin{tabular}{|c|c|c|c|c||}
\toprule
M/N - score& 200 & 400 & 600\\
\midrule
10000 & 0.9494 & 0.9503 & 0.9525\\
20000 & 0.9666 & 0.9689 & 0.9667\\
30000 & 0.9724 & 0.9755 & 0.9778\\
40000 & 0.9754 & 0.9767 & 0.9793\\
50000 & 0.9786 & 0.9807 & 0.9794\\
60000 & 0.9786 & 0.9815 & 0.9783\\
\bottomrule
\end{tabular}
\centering
\qquad
\begin{tabular}{|c|c|c|c|c|}
\toprule
M/N - time& 200 & 400 & 600\\
\midrule
10000 & 2.26 & 3.13 & 3.63\\
20000 & 3.08 & 4.90 & 5.74\\
30000 & 4.95 & 8.54 & 9.22\\
40000 & 5.66 & 9.46 & 10.89\\
50000 & 8.03 & 13.40 & 14.54\\
60000 & 8.24 & 14.15 & 16.84\\
\bottomrule
\end{tabular}
\caption{  Keras -- scores and times}
\label{keras}
\end{figure}


\paragraph*{Extrapolation.}

{
A first idea is to extrapolate a categorical function $f_X \in \{0,\ldots,T\}$ on the training set $Z$ according to 
}
\eqref{equation:EX-00}. However, we observed a rather poor accuracy with this strategy, 
which is not consistent with re-labelling categorical invariance. Instead, let us consider the following partition of unity, 
where $T=10$ is the number of categories referred to in the literature as \textsl{one-hot encoding}: 
\bel{f}
f(x) = \big(f_1(x),\ldots,f_T(x) \big) \in \RR^T, \quad  f_t(x)
= \begin{cases}
1 & \text{ when } L(x) = t, \quad x \in X, 
\\
0 & \text{otherwise.}
\end{cases}
\ee
Here we denote by
$
z \mapsto L(z) \in \{0, \ldots,9 \}
$
 the function returning the classification for any image $z \in [0,1]^D$, and 
 the property $f(z) \in \{0,1\}^T$ holds for the function corresponding to this partition. 
 
 Consider the extrapolation of this function on the set $Z$ using \eqref{equation:EX-00}, that is, in this context
\bel{equation:EX}
f_X(Z) = \Extra_{X,Z}^K f_X(X) = K(Z,X) K(X,X)^{-1} f_X(X) \in \RR^{P \times T}.
\ee
The prediction being then computed as $L_X(z) = \arg \max_t(f_{X,1}(z),\ldots,f_{X,t}(z),\ldots,f_{X,T}(z))$. For this test, only the parameter $M$ is needed. The results are summarized in Figure~\ref{extrapolation}. 
Observe that this extrapolation algorithm is very simple, and has cubic complexity of order ${\mathcal O}(M^3)+{\mathcal O}(MP^2)$
{ where $M$ is the dimension of $X$.}
The table also displays the discrepancy error \eqref{EEPO}, and 
clearly in this case the discrepancy error is similar to the score error. 

The interpretation is the following : as observed earlier, we have the rough estimate of the error in \eqref{equa-erro}
(where $C_1, C_2$ are constants): 
\bel{EMNL} 
\aligned
E(M,M,L) &= \| 1_{L  = L_X } \|_{\ell^1(Z)} 
\leq C_1 \, \| f  - f_{X}  \|_{\ell^1(Z,\ell^\infty(T))} \leq C_2 \,  \| f  - f_{X}  \|_{\ell^1(Z)}   
  \leq  C_2 \, \dK(\mu_X, \mu_Z)\|f \|_{\Hcal^K_X},
\endaligned
\ee
where we use the notation $\ell^p$ to denote the standard (integral or pointwise) discrete norm. 
We can consider here of several optimization choices based on this expression. 
For instance, we can try to reduce $\|f \|_{\Hcal^K_X}$ by picking up a more adapted kernel $K$, which 
an idea closely related to {\sl variance reduction for Monte-Carlo sampling.}
This kernel engineering approach includes {\sl data filtering,}
 and does lead to accuracy improvements, however at a cost of tuning our learning machine to input data which may not be always recommended.  We will not explore this direction here, since we want to perform comparisons with standard methods. A second possibility is to choose the sampling and training set by reducing the distance $\dK(\mu_X, \mu_Z)$. In our example, we prefer to consider that these are fixed data.

\begin{figure}
\centering
\resizebox{\textwidth}{!}
{
\begin{tabular}{|c|c|c|c|c|c|c|c|c|c|c|}
\toprule
M & 1000 & 2000 & 3000 & 4000 & 5000 & 6000 & 7000 & 8000 & 9000 & 10000\\
\midrule
scores & 0.93&0.9492&0.9551&0.9619&0.9649&0.9666&0.9692&0.9707&0.9708&0.9726\\
\midrule
times & 1.26&2.83&5.12&8.53&11.10&16.28&22.41&30.37&40.06&48.68\\
\midrule
$1-\dK(X,Z)$ & 0.9194&0.9384&0.9534&0.9537&0.9610&0.9561&0.9580&0.9586&0.9598&0.9654 \\
\bottomrule
\end{tabular}
}
\caption{  Extrapolation -- scores and times}
\label{extrapolation}
\end{figure}

\paragraph*{Projection.}

We now study the effect of picking up a smaller set $Y \subset X$ in order to {\sl reduce the complexity}
 of the vector machine algorithm. Our projection algorithm is based on interpolating the test-function over a smaller set $Y \in \RR^{N \times D} \subset X \in \RR^{M \times D}$ which is randomly selected among $X$, then
  using $Y$ for extrapolation on the set $Z$ using \eqref{IN}. To summarize, in this context the predictor is 
\bel{equation:EX}
f_X(Z) = \Pibf_{X,Y,Z}^K f_X(X) = K(Z,Y) K(Y,X)^{-1} f_X(X) \in \RR^{P \times T}, 
\ee
where the kernel $K(Z,Y) \in \RR^{P \times N}$. Figure~\ref{interpolation} contains the scores and times for values of $N,M$ chosen in order to keep computational times comparable to our benchmark test. 
{ 
The associated algorithm is quite simple to implement and enjoys the following two features:
(1) For a given computational time, we observe that this projection method achieves better performance in comparison to a crude extrapolation algorithm, however at the cost of tuning the parameters $N$ and $M$.
(2)  This algorithm reduces the algorithmic complexity of the output machine $L(z)$.
}


\begin{figure}
\centering
\begin{tabular}{|c|c|c|c|c||}
\toprule
M/N - score& 1000 & 2000 & 3000\\
\midrule
5000&0.9619&0.9634&0.9632\\
10000&0.9654&0.9701&0.9695\\
15000&0.9686&0.9711&0.9739\\
20000&0.97&0.9735&0.9753\\
25000&0.9713&0.9742&0.9766\\
30000&0.9725&0.9753&0.978\\
\bottomrule
\end{tabular}
\centering\qquad
\begin{tabular}{|c|c|c|c|c|}
\toprule
M/N - time& 1000 & 2000 & 3000\\
\midrule
5000&2.65&5.38&8.62\\
10000&4&7.97&12.74\\
15000&5.34&10.89&17.49\\
20000&6.83&13.79&22.12\\
25000&8.28&16.77&27.06\\
30000&9.70&19.71&31.83\\
\bottomrule
\end{tabular}
\caption{  Projection -- scores and times}
\label{interpolation}
\end{figure}


\paragraph*{Matching algorithm.}

{
The following numerical strategy is now introduced: we apply the projection operator \eqref{equation:EX} as a predictor, but
$Y = \arg \inf_{Y \in \RR^{N\times D}} \dK(\mu_{Y,f_Y}, \mu_{X,f_X})^2$,
where the discrepancy error $\dK(\mu_{Y,f_Y}, \mu_{X,f_X})^2$ corresponds to the discrepancy error of the tensor kernel constructed as the product of the linear regression kernel $K_f$ and the Gaussian kernel.
From the algorithmic point of view, this approach corresponds to matching each category by using \eqref{SDA}, that is, 
$
 Y^p = \arg \inf_{Y^p \in \RR^{N\times D}} \dK(\mu_{Y^p,f_Y^p}, \mu_{X^p,f_X^p})^2$ for 
 $p = 1, \ldots, 10$, 
which significantly reduces the overall algorithmic complexity. Figure~\ref{fitting} displays our results with input data comparable to our benchmark test. In Figure~\ref{projfitting}, we show a comparison with the projection algorithm for comparable inputs.
}

\begin{figure}
\centering
\begin{tabular}{|c|c|c|c|c||}
\toprule
M/N - score& 200 & 400 & 600\\
\midrule
10000&0.9396&0.9478&0.9509\\
20000&0.9433&0.9529&0.9556\\
30000&0.9416&0.9527&0.9581\\
\bottomrule
\end{tabular}
\centering\qquad
\begin{tabular}{|c|c|c|c|c|}
\toprule
M/N - time& 1000 & 2000 & 3000\\
\midrule
10000&36&47&58\\
20000&117&143&160\\
30000&251&286&312\\
\bottomrule
\end{tabular}
\caption{  Matching algorithm -- scores and times}
\label{fitting}
\end{figure}
\begin{figure}
\centering
\begin{tabular}{|c|c|c|c|c||}
\toprule
M/N - score& 200 & 400 & 600\\
\midrule
10000&0.911&0.928&0.9357\\
20000&0.9153&0.9356&0.9437\\
30000&0.9182&0.9379&0.9457\\
\bottomrule
\end{tabular}
\centering
\qquad
\begin{tabular}{|c|c|c|c|c|}
\toprule
M/N - time& 1000 & 2000 & 3000\\
\midrule
10000&0.86&1.29&1.75\\
20000&1.38&1.98&2.62\\
30000&1.79&2.60&3.47\\
\bottomrule
\end{tabular}
\caption{  Projection algorithm -- scores and times}
\label{projfitting}
\end{figure}


\section{An algorithm for the Fokker-Plank-Kolmogorov system in mathematical finance}
\label{sec-44}

\subsection{The fundamental equations}

\paragraph*{Stochastic differential equation.}

Let us begin with a description of the fundamental equations of mathematical finance, that is, 
the Fokker-Planck and Kolmogorov equations. We are interested in Markov-type stochastic processes, denoted by
$t \mapsto X_t \in \RD$ and governed by a stochastic differential equation (SDE).
Specifically, we want to solve the {\sl forward} initial value problem
\bel{equa-SED}
\aligned
& dX_t = r(t,X_t)dt+\sigma(t,X_t)dB_t, 
\qquad
\qquad
 X_t|_{t=0} = X_0. 
\endaligned
\ee
Here, the initial data $X_0 \in \RR^D$ is a prescribed random variable and $B_t \in \RR^D$ denotes a ($D$-dimensional, with independent components) Brownian motion.
We are also given a vector field $r \in \RR^D$ representing the drift (or risk-free rate), 
together with a matrix-valued field $\sigma = \sigma(t,x)$ 
representing the random diffusion (or volatility). 
In our presentation we consider $\sigma$ and $r$ as given, while in specific applications it might be required to perform  
a calibration first and solve an inverse problem based on market data. A description of such a calibration algorithm can be found in \cite{PLF-JMM-2}.
 

\paragraph*{Fokker-Planck equation.}

We are going to work with the description based on the partial differential equation associated with the SDE \eqref{equa-SED}. We introduce the linear operator  
\bel{equa--Lcal} 
\Lcal \mu= \nabla \cdot (r \mu) + \nabla^2 \cdot (A \mu), \qquad A= {1 \over 2} \sigma \sigma^T. 
\ee
and consider the corresponding convection-diffusion equation. 
{ 
Here, as is standard for PDEs we denote by 
$\nabla$ the gradient operator and by $\nabla \cdot$ the divergence operator, while
 $\nabla^2= (\del_i \del_i)_{1 \leq i,j \leq D}$ denotes the Hessian operator. 
 }
Given any parameter $(s,y) \in [0, +\infty) \times \RD$, we denote by $(t,x) \mapsto \mu(t,x; s,y)$ (defined for $t\geq s$ and $x \in \RD$) the {\sl probability density measure} of the random variable $X_t$, {\sl knowing} the value $X_s = y$ at the time $s$. By definition, this measure is the fundamental solution, defined 
from the base point $(s,y)$, of the {\sl Fokker-Planck problem} 
\bel{equa--FP}
\del_t \mu - \Lcal \mu = 0, \quad t \in [s, +\infty),
\qquad \mu|_{t=s} = \delta_{x=y}. 
\ee 
The initial data in \eqref{equa--FP} is thus posed at the time $t=s$ and coincides with the Dirac measure $\delta_{x=y}$ at the point $x=y$. 

By construction of the Fokker-Planck equation, the expectation of the random variable 
with respect to an arbitrary test-function $\varphi$ is given as the corresponding moment of the measure $\mu$. 
That is, the relation 
\be
\int_{\RD} \varphi(t,x) \, d\mu(t, x; 0, X_0) = \Ebf^{X_t} (\varphi(t, \cdot) | X_0)
\ee
holds between the measure $\mu$ and the initial data $X_0$.  
We also observe that the solution to \eqref{equa--FP} is more naturally understood in the sense of distributions when the initial data is a measure. The solution $\mu$ is then a probability measure since the total mass is conserved in time, i.e. 
$
\int_{\RD} \mu(t,x; s,y) dx = 1.  
$
Finally, we recall that the equation \eqref{equa--FP} also models the Brownian motion of particles in plasma physics.

It will be convenient to distinguish between two cases of interest in operational problems :
\bei

\item The {\bf martingale Fokker-Planck equation,} corresponding to the purely diffusive Fokker-Planck equation without transport part, that is, with $r \equiv 0$ in \eqref{equa-SED}.  

\item The {\bf general Fokker-Planck equation,} which includes both hyperbolic and parabolic parts, in which $\sigma$ and $r$ are prescribed functions.

\eei
%


\paragraph*{Kolmogorov equation.}

The backward dual problem defined now provides us with the natural standpoint in mathematical finance.  
Taking the dual of the Fokker-Planck equation leads us to the {\sl Kolmogorov equation} which is  
the central equation of interest, governing for instance for the price of option values, and is referred to as the {\sl Black and Scholes equation}. 
It concerns an $M$-vector-valued unknown denoted by $P \mapsto P(t,x) = P (t,x; T, P _T)$,
defined for all times $t \leq T$ and $x \in \RD$, and is formulated as the following {\sl backward problem.} 
Given a final time $T$ and any (continuous and $\mu$-integrable) final data $P _T$, we seek the solution to  
\bel{equa--KE}
\aligned
& \del_t P  - \Lcal^* P  =0,  \quad t \in [0, T], 
\qquad
\qquad 
  P|_{t=T} = P _T, 
\endaligned
\ee
where $\Lcal^*$ denotes the dual of the Fokker-Planck operator, given by 
$
\Lcal^* P = - r \cdot \nabla P  + A \cdot \nabla^2 P. 
$

In mathematical finance, the vector-valued function $P  \in \RR^M$ represents a portfolio of $M$ instruments and $M$ is typically a large integer. We recall the following terminology. 

\bei 

\item The final data  $P _T$ is called the {\sl payoff} of the instruments, whose underlyings are described by the random variable $X_t$ satisfying the stochastic differential equation \eqref{equa-SED}. 

\item The Kolmogorov equation \eqref{equa--KE} determines the so-called {\sl forward values} (or price) $P(t,x)$ of the porfolio at any later time $t \in [0,T]$. 

\item In addition, the value at the time $t=0$ is called the {\sl fair value} $P_0 = P|_{t=0}$ of the portfolio. 

\eei 

\noindent Solving the Kolmogorov equation for a given portfolio of instruments allows us to determine not only its price,  
 but also the whole of the {\sl fair value surface} $(t,x) \mapsto P (t,x)$ (for all $t \in [0, T]$ and $x \in \RD$). 
This function is important in many applications since standard risk measures are often determined from this surface; for instance, this is the case of risk measures of internal or regulatory nature, as well as of optimal investment strategies. 


\paragraph*{Martingale case.}  

Let us consider here the martingale case and, for clarity in the presentation, let us repeat here the main equations. 
This case consists in solving the Fokker-Planck equation \eqref{equa--Lcal} and the Kolmogorov equation \eqref{equa--KE} 
without the advection part, that is, 
\bel{equa-mart}
 dX_t = \sigma(t,X_t)dB_t, \qquad
  t \ge s \ge 0,
\ee
with initial conditions $X_0 \in \RR^D$. The Fokker-Planck equation \eqref{equa--Lcal} for
 the probability measure $\mu(t,\cdot)$ reads (for $t\ge 0$)
\bel{equa--Lcalmart} 
\del_t \mu= \nabla^2 \cdot (A \mu), 
\qquad A= {1 \over 2} \sigma \sigma^T,
\ee
with initial conditions $\mu(0,\cdot) = \delta_{X(0)}$.   
In this martingale context, we then solve  the Kolmogorov equation \eqref{equa--KE}, that is,
 for some continuous and $\mu$-integrable function $P$
\bel{equa--KEmart} 
 \del_t P  -  A \cdot \nabla^2 P =0,  \quad t \in [0, T], 
\qquad
\qquad 
  P|_{t=T} = P _T. 
\ee

We will use the following definition associated with any martingale process $t \mapsto X_t \in \RR^D$ (for $t \ge 0$). 
In the range $0 \le s \le t$ and $x,y \in \RR^D$, we define the {\sl transition probability operator}
 $\Pi = \Pi(s,t,x,y)$ as the solution to the Kolmogorov equation in the variables $(t,y)$ 
\bel{equa--Kolm} 
\del_t \Pi = A \cdot \nabla_y^2 \Pi, \qquad \Pi(s,t,x,\cdot) = \delta_{x}, 
 \qquad 
 t \ge s \ge 0, \quad x \in \RR^D.
\ee
It is not difficult to check the following property. If $t \mapsto X_t \in \RR^D$ (for $t \ge 0$)
is a  martingale process and $\mu(t,\cdot)$ denotes its density measure, using transition operator $\Pi$ we have 
\bel{Pixy}
P(s,x) = \Pi(s,t,x,\cdot) \ast P(t, \cdot) = \EE^{X_t}\Big( P(t,\cdot) | X_s = x\Big), 
\ee
where $\ast$ denotes the convolution of two functions. The first identity holds since, by definition, the operator $\Pi(t,s, x ,y)$ is the fundamental generator of the Kolmogorov equation, while the second identity is a consequence of Feynman-Kac's theorem.


\subsection{The transport mesh-free method for mathematical finance}
\label{TMMF}

We now present our algorithm and solve numerically the equations \eqref{equa--FP} and \eqref{equa--KE}, as follows. The method was first outlined in \cite{PLF-JMM-2} and depends upon the a priori choice of an admissible kernel $K$ of the type described in Section~2. We are going to compute a transport map, denoted below by $t\mapsto S(t,\cdot)$, which is naturally associated with the probability density $\mu$ of the stochastic process $X_t$ satisfying \eqref{equa-SED} and associated with a random initial data $X_0$.  


\paragraph*{Step 1: Forward computation.} 

Our first task is to determine the (almost) sharp discrepancy sequence that is most relevant for the simulation of the SDE problem \eqref{equa-SED}. We introduce a semi-discrete scheme for the the Fokker-Planck equation \eqref{equa--FP}, 
in which the time variable is kept continuous.
This scheme provides us with an approximation of the solution $\mu$ as a sum of Dirac measures, namely   
\be
\mu_{S}(t) = \frac{1}{N} \sum_{n=1}^N  \delta_{S^n(t)},
\qquad 
S(t) = (S^1, \ldots, S^N)(t), 
\ee
in which $t \mapsto \big( S^n(t) \big)_{n=1, \ldots, N}$ in $\RD$ is a suitable family of  $N$ ``moving particles''.

Specifically, a (small) approximation parameter $\eps \in (0,1)$ being fixed, we solve the coupled system of ordinary differential equations  (for $t \geq t_0>0$) 
\bel{equa--SDS}
\aligned
\frac{d}{dt} S & = r_S + \big( \nabla_S \big)^T A_S \in \RR^{N \times D},  
\qquad
\qquad
 S\big|_{t= t_0}  = (X_0, X_0, \ldots, X_0) + \sqrt{t_0} \, A(0, X_0) \Ncal(0,1), 
\endaligned
\ee
in which the function $r$ and $A$ in \eqref{equa--FP} are evaluated pointwise along the trajectories of the moving particles:
\be 
r_S(t) = \big(  r(t,S^n(t)) \big), 
\qquad  
A_S(t) = \big( A(t,S^n(t)) \big)_{n=1, \ldots, N}. 
\ee
Recall that the discrete differentiation $\big( \nabla_S \big)^T$ and, therefore, $\big( \nabla_S \big)^T A_S$ were defined in Section~\ref{section-21}. 
In \eqref{equa--SDS}, $X_0$ stands for a specific realization of the random initial data,
 and in order to prevent the possibility of colliding particles, we have added a (vector-valued, centered, and of unit variance) i.i.d. variable\footnote{By definition, i.i.d. variables 
are independent and identically distributed random variables.}
  $\Ncal(0,1)$.  


\paragraph*{Step 1. Stability.}

The dynamics of particles is determined by the coupled system \eqref{equa--SDS}, and we have defined an approximation of the solution $\mu(t) \simeq \mu_S(t)$ to the forward Fokker-Planck problem \eqref{equa--FP}. In the next section, we will check that this scheme is consistent and stable and, in fact, in the long time limit
the sequence $S(t)$ approaches a sharp discrepancy sequence for the kernel $K$. 
The above property has an important consequence, that is, we can rely on the framework developed in Section~2. In order to evaluate the accuracy of our strategy, at each discrete time we 
can compute the discrepancy error \eqref{equation-EE} and obtain the explicit error estimate
\bel{equa--FPE} 
\Big| \int_{\RR^D} \varphi(x)d\mu(t) - \frac{1}{N}\sum_{1 \leq n \leq N} \varphi(S^n(t)) \Big| 
\leq \dK(\mu(t), \mu_S(t)) \, \|\varphi\|_{\Hcal^K(\RD)}, 
\ee
valid for any moment of the measure $\mu$ and any time $t$
($\varphi$ being an arbitrary test-function). 
Based on our analysis of the sharp discrepancy error, we can compare $\dK(\mu(t), \mu_S(t))$ with the minimum value denoted (earlier on by $\dK(\mu, \delta_{\Ybar})$ in Section~2) 
and, in turn, {\sl explicitly check} the accuracy of the numerical solution.


\paragraph*{Step 2. Backward computation.} 

At this stage, we have computed an approximation of the sharp discrepancy sequence which is relevant (for a given realization of the initial data), and we are in a position to solve the backward Kolmogorov problem \eqref{equa--KE}, using $t \mapsto S^n(t)$ (with $n= 1, \ldots, N$) as a {\sl moving grid}. 
We determine an function $t \mapsto P_S(t) \in \RR^{N \times M}$ 
which is expected to be an approximation of the pointwise values 
$\big( P(t, S^n(t)) \big)_{1 \leq n \leq N}$ of the  Kolmogorov solution. 
As we further explain below, it can be computed as follows, between any two arbitrary times $s < t$ (for instance $0=s<t=T$), 
\bel{equa-KED} 
P_S(s) =  \Pibf_S^{ (t,s)}P_S (t),  
\qquad 
\Pibf_S^{ (t,s)}= \big( \pi^{ (t,s)}_{n,m} \big)_{1 \leq n,m \leq N}. 
\ee
Here, the matrix $\Pibf_S^{ (t,s)} \in \RR^{N \times N}$ is provided {\sl explicitly} (see below) 
and is nothing but the discrete fundamental operator associated with the Kolmogorov equation. Hence, 
this step of our numerical algorithm reduces to {\sl computing the inverse of a matrix.} 
The consistency and convergence of the method are discussed below. 

\paragraph*{Step 2. Stability.} 

Interestingly, in the present context we can regard the matrix $\Pibf_S^{ (t,s)}$ as a {\sl Markov-chain process:}  $\pi^{ (t,s)}_{n,m}$ is the probability that the stochastic process will jump from the sharp discrepancy state $S^n(t)$ (at the time $t$)
 to the sharp discrepancy state $S^m(s)$ (at the time $s$). Our numerical algorithm provides this structure in a natural way, 
 and $\pi^{ (t,s)}_{n,m}$ turns out to be a bi-stochastic matrix (having each rows and each column summing to $1$). 
This step of the algorithm is a martingale process which takes into account the diffusive part $A$ of the SDE problem \eqref{equa-SED} but with the drift part $r$ {\sl suppressed} (since it was already handled in the first step). 
Importantly, the numerical solution enjoys the error estimate associated with the kernel $K$, that is, 
\be
 \Big| \int_{\RR^D} P (t,\cdot)d\mu(t,\cdot) - \frac{1}{N}\sum_{1 \leq n \leq N} P (t,S^n(t)) \Big| 
  \leq 
\dK(\mu(t), \mu_S(t)) \, \|P (t,\cdot)\|_{\Hcal^K(\Rd)}, 
\ee
which allows us to quantitatively determine, once more, the {\sl accuracy of the numerical solution.} 
{
We can also deal with partial derivative operators and, for instance, our framework allows us to compute forward sensitivities, that is, $\nabla P (s) \text{ is approximated by } \nabla_y P (t,y^n(t))_{1 \leq n \leq N}$. 
This allows us to determine, for instance, hedging strategies, as we discuss later in this section. 
Motivated by fluid and material dynamics, we could also treat more complex operators such as the Hessian operator or the Helmholtz-Hodge decomposition.
}

 
\paragraph*{Numerical illustration: the SABR model.} 

We can illustrate a feature of our algorithm by considering the example of the (shifted) SABR model (see \cite{AKS2015} and the references therein), described by the following two stochastic differential equations with unknown variables $X=(F, \alpha)$: 
\begin{equation}
\label{eq-modelSABR}
\begin{pmatrix}
d F_t 
\\ 
d \alpha_t
\end{pmatrix} 
= \rho \, \begin{pmatrix}
\alpha_t (F_t+s)^\beta & 0 \\ 
0 & \nu \alpha_t
\end{pmatrix} 
\,
 \begin{pmatrix}
 dW_t^1 \\ 
 dW_t^2
\end{pmatrix},
\ee
supplemented with prescribed random initial conditions $X_0 = (F_0,\alpha_0)$ at the time $t=0$.  Here, $\beta \in [0,1]$ is a parameter representing the so-called {\sl constant elasticity of variance} (CEV) and the constant $\nu \geq 0$ is the volatility parameter, while $W^1_t, W^2_t$ are independent Brownian motions, while the so-called correlation matrix  $\rho$ is also prescribed.  
Let us consider the transported Mat\'ern kernel, described for instance in \cite{PLF-JMM-estimate}. 
In Figure~\ref{SABRden}, we plot our approximation of the sharp discrepancy sequence associated with the SABR model. We use $N=200$ points and the parameters $F_0=0.03$, $\alpha_0 = 0.1$, $\nu = 0.1$, $\beta = 1$, 
 and $\rho_{12} = \rho_{21} = 0.5$ (with $\rho_{11} = \rho_{22} = 1$ for the correlation matrix). 
We plot the set of points \(\big( y^1(t),\ldots, y^N(t)\big)\), where the $y$-axis represents the volatility process $\alpha_t$, and $x$-axis denotes the interest rates $F_t$. We observe that the initial support (an approximation of a Dirac measure) spread throughout the computational domain and expands as the time increases, according to the transport-diffusion process described by \eqref{eq-modelSABR}. 
 

\begin{Figure}
\includegraphics[width=0.32\linewidth]{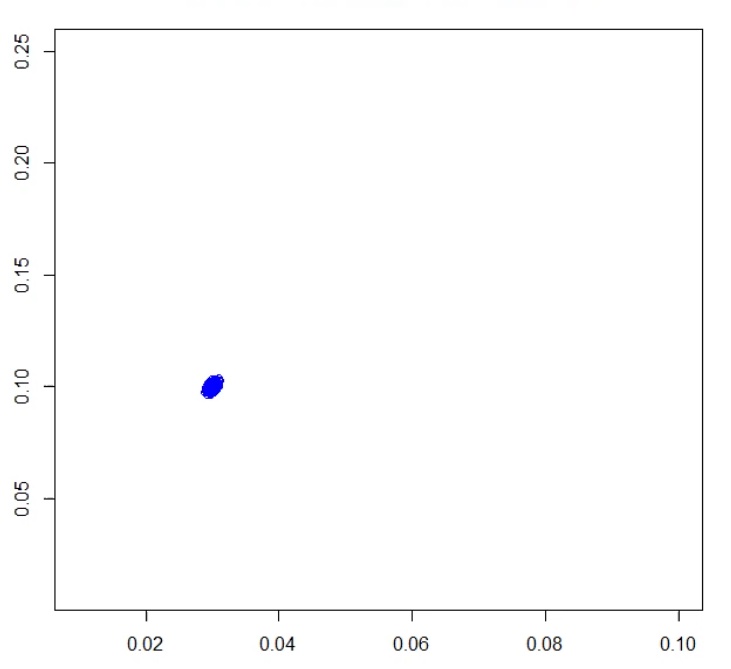} \includegraphics[width=0.32\linewidth]{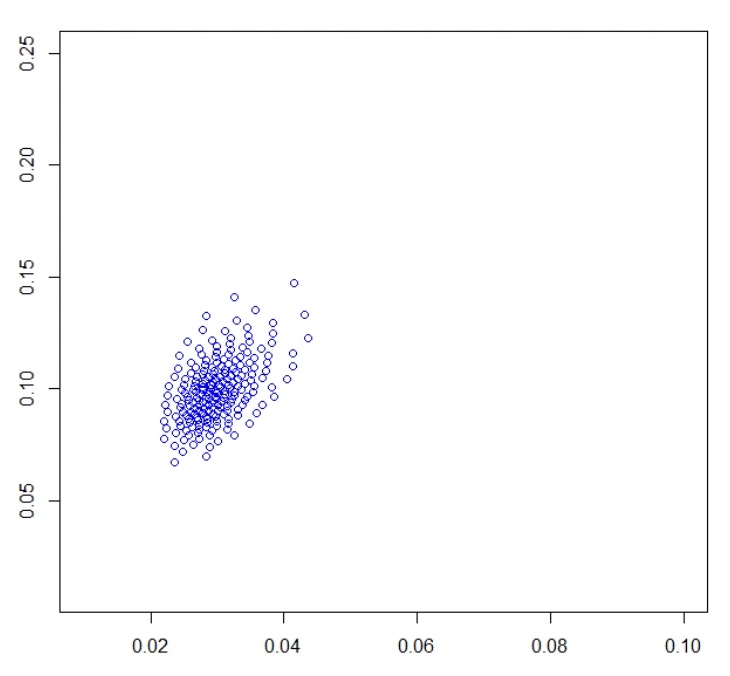} 
\includegraphics[width=0.32\linewidth]{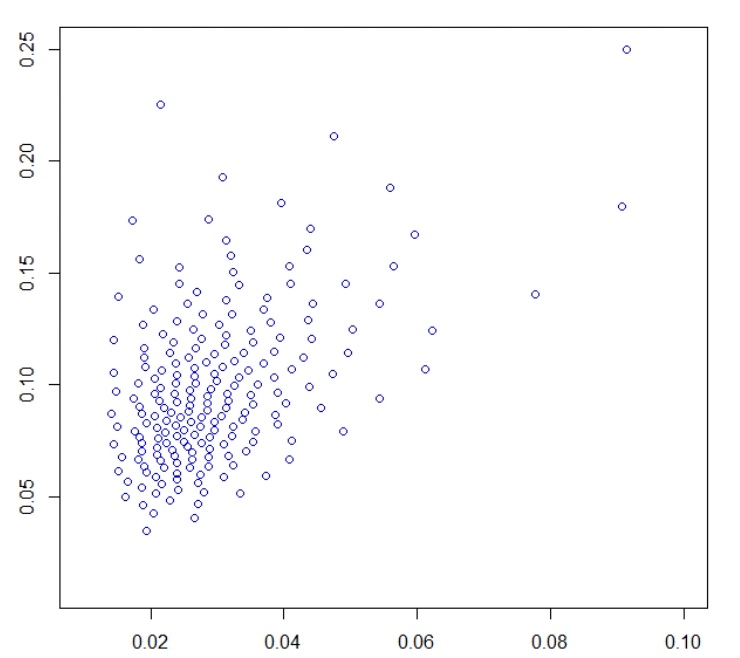}

\centerline{(a) horizontal $y^1$,  vertical $y^2$ \qquad \qquad (b)  horizontal  $y^1$, vertical $y^2$ \qquad (c)  horizontal  $y^1$, vertical $y^2$ \qquad}
\captionof{figure}{  SABR model at times $t=0.02, 2, 12$ with $N=200$ points.}
\label{SABRden}
\end{Figure}



\subsection{Analysis of the semi-discrete scheme for the Fokker-Plank equation}

\paragraph*{Main statement for this section.}

The semi-discrete scheme \eqref{equa--SDS} enjoys many properties (global existence, stability, convergence) which we now state.  We work with the following {\sl modified Fokker-Planck problem} 
\bel{equa--FP-mod}
\aligned 
& \del_t \muexa - \Lcal \muexa = 0, \qquad 
&&t \in [t_0, +\infty), \, x \in \RD, 
\\
& \muexa|_{t=t_0} = \mu_{S_0^\star}, 
\qquad 
&& S_0^\exact = (X_0, X_0, \ldots, X_0) + \sqrt{t_0} \, A(0, X_0) \Ncal(0,1), 
\endaligned
\ee 
whose (exact) unknown is denoted by $\muexa= \muexa(t,x)$ and for which the initial data is posed at some 
time $t=t_0$ and randomly perturbed by a Brownian motion (as explained earlier). 
We impose a uniform ellipticity condition in \eqref{equa-assume} below.
 Thanks to the uniform parabolic property we can restrict attention to sufficiently regular solutions. 

An important insight for our theory is provided by the transported version of the evolution equation. By direct considerations, the first step of the algorithm is seen to enjoy the following properties. 
 
\begin{claim}[Semi-discrete algorithm for the Fokker-Planck equation]
 \label{SD2}
  
Consider the Fokker-Plank problem \eqref{equa--FP-mod} under the following ellipticity and linear growth conditions
(with $\lambda, \alpha_-, \alpha_+ >0$): 
\bel{equa-assume}
| r(t,x) \cdot x | \leq \lambda  |x|^2,  
\qquad 
\alpha_-  |x|^2 \leq x^T \cdot A(t,x) \cdot x \leq \alpha_+  |x|^2, 
\qquad t \geq 0, \, x \in \RD.
\ee
Then the semi-discrete scheme \eqref{equa--SDS} defines a sequence of moving particles
 $S(t) =(S^n(t))$ 
together with an associated measure $\mu_S(t) = \frac{1}{N} \sum_n \delta_{S^n(t)}$ 
which determines a convergent approximation of the solution $\muexa(t)$ to \eqref{equa--FP-mod}, as follows. 

\bei
\item {\bf Sup-norm estimate:} 
$\sup_{ t \geq 0} e^{-\lambda t} \sum_{n=1}^N   |S^n(t)|^2 \leq 
\sum_{n=1}^N  |S^n(0)|^2$, 
which is an analogue of the second moment $\int |x|^2 d\muexa$ for solutions to the Fokker-Plank equation. 

\item {\bf Balance law of momentum:} 
$\frac{d}{dt}\sum_{n=1}^N S^n(t) = \sum_{n=1}^N  r(t,S^n(t))$ for all $t \ge 0$
and, in particular, the total momentum is constant in time when the drift vanishes identically. 

\item {\bf Self-adjointness property.}  The scheme \eqref{equa--SDS} can be cast in the symmetric form
$\frac{d}{dt} S = r_{S} + B_S S,
$
in which $B_S = B_S(t)$ is a positive-definite symmetric matrix. 

\item {\bf Error estimate.} For any relevant test-function $\varphi$ and any time $t \geq 0$ one has 
\be
\bigg|\int_{\RR^D} \varphi(x) d\muexa(t,x) - {1 \over N}  \sum_{n=1}^N  \varphi(S^n(t)) \bigg| 
\leq 
\dK(\muexa(t), \mu_S(t)) \,  \|\varphi \|_{\Hcal^K(\RD)}.
\ee

\item{\bf Convergence property:} 
$\lim_{N \to +\infty} \dK(\muexa(t), \mu_S(t)) = 0$ for all $t \geq t_0$.

\eei

\end{claim}


\subsection{A semi-discrete scheme for the Kolmogorov equation}
 
We present here some basic estimates enjoyed by our scheme and, in addition, numerical experiments demonstrate that it also enjoys better convergence rates in comparison to classical finite difference schemes.  
For simplicity in the presentation and without genuine loss of generality we assume that $M=1$, that is, the main unknown $P$ of the Kolmogorov equation is now scalar-valued. 
At this stage, we have determined a map $t\in [0,T] \mapsto S(t) \in \RR^{N\times D}$ by the algorithm in Claim~\ref{SD2} and we can thus turn our attention to solving the (transported version of the) Kolmogorov equation \eqref{equa--KE}. 
In fact, since the transport part has been treated in the first step, we can focus our attention on the diffusion part, so without loss of generality we assume that the drift $r$ vanishes identically and we study the martingale case \eqref{equa--KEmart}. 
Hence, the Kolmogorov equation reads 
\bel{equa-Kolmo4}
\del_t P - \Lcal^* P = 0, \qquad \Lcal^* P = A \cdot \nabla ^2 P.
\ee
Our backward semi-discrete scheme for computing the solution $P=P(t,x)$ for $x \in \RN$ and $t \in [0,T]$ reads 
\bel{KDS}
\frac{d}{dt} P = - B_S P, 
\qquad
\quad B_S  =  (\nabla^K_S)^T A_S (\nabla^K_S), 
\ee
supplemented with the initial condition $P|_{t=T} = P _T(S(T))$ prescribed at the time $T$.
Recall that the map $t \mapsto S(t)$ is computed by the semi-discrete scheme \eqref{equa--SDS}, which can be expressed, 
as $\frac{d}{dt}S = (\nabla_S^T) A_S$. 
On the other hand, the scheme for $P$ can be written in the compact form 
\bel{mart}
\frac{d}{dt} P = - \Big((\nabla^K_S)^T \big( (\nabla_S^K)^{-T}\frac{d}{dt} S \big)  (\nabla^K_S) \Big)  P,
\ee
where $(\frac{d}{dt} S)  (\nabla^K_S) \in \RR^{N \times N}$ denotes the contraction of the discrete operator
 $(\nabla^K_S) \in \RR^{D \times N \times N}$ and the matrix $(\frac{d}{dt} S) \in \RR^{ N \times D}$ in the first  and second components.

\begin{claim}[Semi-discrete algorithm for the Kolmogorov equation]
\label{KSD} 
 
Under the assumption \eqref{equa-assume}, consider $t \in [t_0, T]
 \mapsto S(t) \in \RR^{N \times D}$ a solution to transported scheme \eqref{equa--SDS}, with no drift, that is $r\equiv 0$, modeling a martingale process. Consider the scheme \eqref{KDS}. 
\bei
\item 
{\bf Consistency.} The scheme \eqref{KDS} is consistent with a transported version of the Kolmogorov equation \eqref{equa-Kolmo4}.
(Cf.~\eqref{KTS} below.) 

\item {\bf Sup-norm estimate:} 
$\sup_{s \in [0,t]} \sum_{n=1}^N  (P^n(s))^2 \leq \sum_{n=1}^N  (P^n(t))^2$ for all $t \in [0, T]$, 
which is an analogue of the decay property of the moment $\int P^2 d\muexa$ for exact solutions. 

\item  {\bf Moment conservation property:} 
$
\sum_{n=1}^N P^n(t) =  \sum_{n=1}^N P^n(T)$ for all $t \in [t_0, T]$, 
which is an analogue of the moment $\int P d\muexa$ for exact solutions. 

\item {\bf Stochastic property.} 
If $\Pi = \Pi(s,t) \in \RR^{N \times N}$ denotes the forward generator of \eqref{KDS}, that is, the solution to 
\bel{PI}
\frac{d}{dt} \Pi(s,t) =  B_S(t) \Pi (s,t), 
\qquad \Pi(s) =I_{N \times N}, 
\quad t \in [s, T],  
\ee
then one has 
$
P(s) = \Pi(s,t) P(t), \qquad s \in [t_0,t].  
$
The matrix $\Pi(s,t)$ is a stochastic matrix provided $\nabla_S^T A_S \nabla_S$ is essentially non-negative.
\eei
\end{claim}

The latter condition can be ensured by restricting attention to a class of kernels, which should be in agreement with the dissipative aspects of the Kolmogorov operator.
We can view $\Pi(s,t)$ as a {\sl transition probability matrix} between times $s$ and $t$. 
We recall that an \emph{essentially non-negative matrix} $M = (m_{ij})$ is a matrix with
 positive off-diagonal elements $m_{ij} \ge 0, i\neq j$. We recall also the following result (see \cite{Varga}): a matrix $M$ is an essentially non-negative matrix if and only if $\exp(tM)$ is positive.
The stochastic property is an important property to check for in mathematical finance applications, and the 
associated scheme \eqref{KDS} satisfies a discrete maximum principle provided this condition is fulfilled. 
The condition 
$\big( (\nabla_S)^T A (\nabla_S) \big)_{n,m} \geq 0$ for all $n\ne m$ (with $1 \leq n,m \leq N$)
 is a technical condition that restricts the possible shape of the basis functions.

%

\begin{proof}[Proof of Claim \ref{KSD}] 
 
{\sl Consistency of the approximation scheme.} 
We consider the following weak formulation, for any solution $P$
to the Kolmogorov equation \eqref{equa-Kolmo4} and any measure solution $\mu(t,\cdot)$ 
to the Fokker-Planck equation \eqref{equa--FP-mod} without drift term (that is, $r \equiv 0$): 
\be
	\frac{d}{dt} <\mu, P>_{<\mathcal{D}',\mathcal{D}>} = -< A \cdot \nabla^2 P, \mu >_{<\mathcal{D}',\mathcal{D}>} + < P,  \nabla^2 \cdot (A\mu) >_{<\mathcal{D}',\mathcal{D}>} = 0,
\ee
that is,  
\be
	0 = \frac{d}{dt} \int P(t,\cdot) \mu(t,\cdot) = \frac{d}{dt} \int (P \circ S)(t,\cdot) dx. 
\ee
Here, $S(t,\cdot)$ is a transport map of $\mu(t,\cdot)$, solving  the equation without drift ($r\equiv 0$). 
We obtain 
\bel{KTS}
	\frac{d}{dt} (P \circ S) = (\del_t P) \circ S + \del_t S (\nabla P)\circ S =  (\del_t P) \circ S + \nabla \cdot \Big( (A \circ S) \mu \Big) (\nabla P)\circ S  = 0, 
\ee
and $\nabla \cdot \Big( (A \circ S) \mu \Big)$ is consistent with the operator $\nabla_S^T A_S$, and we conclude that the semi-discrete scheme
$
	(\del_t P) \circ S = -\nabla_S^T A_S \nabla_S P, \quad P \in \RR^{N}
$
is consistent with the transported Kolmogorov equation.
This semi-discrete scheme is mass-conservative, that is, 
$$
\frac{d}{dt} <P,1_N>_{\ell^2} = -\la A_S \nabla_S P,\nabla_S 1_N \ra_{\ell^2} = 0,
$$
as soon as $\nabla_S 1_N = 0$, that is fulfilled as soon as a polynomial basis is added.

\vskip.2cm

{\sl Global existence.} The energy bound
$
\frac{d}{dt} \|P\|^2_{\ell^2} = -\la A_S \nabla_S P,  \nabla_S P \ra_{\ell^2} \le 0
$
as soon as the matrix field $A_S$ is positive definite, hence leading to the estimate
$
	\|P(s)\|_{\ell^2} \leq \|P(t)\|_{\ell^2}, \quad s \leq t,
$
and the scheme is a stable contracting scheme for the discrete $\ell^2$ norm.

\vskip.2cm

{\sl Stochastic property.}
A stochastic matrix $M\in \RR^{N\times N}$ is a positive matrix $m_{i,j}\ge 0$, for which $\sum_j m_{i,j} = 1$, 
that is, 
the vector $1_N$ is an eigenvector associated with the eigenvalue $1$. 
The same argument as the one for the mass conservation above applies: 
$$
\frac{d}{dt} \Pi 1_N =-\la A_S \nabla_S \Pi,\nabla_S 1_N \ra_{\ell^2} = 0. 
$$
The positivity $m_{i,j}\ge 0$ follows from \eqref{PI}, from which we deduce
$
\Pi \simeq \exp(\nabla_S A_S \nabla_S). 
$
Provided the matrix $\nabla_S A_S \nabla_S)$ is essentially non-negative, it is positive. Obviously $\frac{d}{dt} \pi_{n,n} \leq 0$, $\pi_{n,n}(0) =1$, but remains positive.
\end{proof}
 

\subsection{Fully discrete, monotone, and entropy dissipative schemes}

\paragraph*{Reduction to a martingale process.}

We now present a full discretization of the Fokker-Planck equation \eqref{equa--Lcal} and 
the Kolmogorov equations \eqref{equa--KE}, which are based on the semi-discrete schemes \eqref{equa--SDS}
 and \eqref{equa-KED}. 
 We introduce a time discretization $0=t^0<t^{\epsilon}<t^1 \ldots<t^{N_t}=T$ 
 with forward increments $\delta^n = t^{n+1}-t^n$. 
 The discrete solutions are represented in a matrix form as 
 $S^n\simeq S(t^n) \in \RR^{N_S \times D}$, $P^n\simeq P(t^n) \in \RR^{N_S \times D_P}$, $A^n\simeq A(t^n,S^n) \in \RR^{N_S \times D \times D}$, etc. 
Consider the scheme \eqref{equa--SDS}, and observe that it consists of a hyperbolic part and a (martingale) diffusive part, so at each time step $t^n$ we split \eqref{equa--SDS} as follows.
 The hyperbolic part is straightforwardly integrated with a classical ODE discretization and we denote 
 by $V^{n+1}$ the corresponding discrete solution produced at the time $t^{n+1}$
 which is an approximation to
\bel{equa--FP-H}
\frac{d}{dt} V = r_{V}, \quad V^n = S^n,\quad t^n \le t \le t^{n+1}. 
\ee
Once $V^{n+1}$ is computed, the numerical scheme for the martingale part is written down, and is considered with the initial condition $V^{n+1}$ at the time $t^n$ while $r_S$ can be replaced by zero, 
that is, we solve 
\bel{equa--FP-M}
\frac{d}{dt} S = B_S S, \quad  B_S  = -\big( \nabla_{S} \big)^T A_{S} \big( \nabla_{S} \big)
\ee
From now on, in order to present our fully discrete scheme we can thus 
focus on the martingale case without discussing further on the hyperbolic part.


\paragraph*{Fully discrete algorithm.}

To compute the solution to \eqref{equa--FP-M}, we first rely on a simple Euler scheme 
 at an intermediate time $t^{n+1/2}$: 
\bel{Euler}
S^{n+1/2} =S^{n} +  \sqrt{(1/2) \, A^n \delta^n} \, \eps_{\mathcal{N}}, 
\ee
where we have the following properties. 

\begin{itemize}
\item We recall that $A^n \in \RR^{N_S\times D \times D}$  represents a field of symmetric positive definite matrix. The square root of $(1/2) A^n \delta^n$ is defined by taking the principal positive square root of each matrix, 
therefore
 $\sqrt{(1/2) \, A^n \delta^n}  
 = \big( \sqrt{(1/2) \, A^n(t^n,S^n,S^n) \delta^n} \big)_{n = 1, \ldots, N_s} \in \RR^{N_S\times D \times D}$.

\item $\eps_{\mathcal{N}} \in \RR^{N_S \times D}$ is selected once for all as a given sampling of the $D$-dimensional normal law $\mathcal{N}^D(0,1)$, and 
we pick up $\eps_{\mathcal{N}}$ as an approximation of a sharp discrepancy sequence of $\mathcal{N}^D(0,1)$;
 see \eqref{equa--SDS}. 
 
 \item $\sqrt{(1/2) \, A^n \delta^n}  \, \eps_{\mathcal{N}}$ stands for the multiplication 
 $\big (\sqrt{ (1/2) \, A^n(t^n,S^n,S^n) \delta^n} \, \eps_{\mathcal{N}}^n\big)_{n = 1, \ldots, N_s}$, 
 which belongs to $\RR^{N_S \times D}$ and thus is a vector field.

\end{itemize}
We now solve \eqref{equa--FP-M}. Observe that a Crank Nicolson scheme for the approximation of this equation would produce, at time $t^{n+1}$, 
\bel{CN}
 S^{n+1}= \big( \Id - \delta^n B_{S^{n+1/2}} \big)^{-1} \, \big( \Id + \delta^n B_{S^{n+1/2}} \big) \, V^n. 
\ee
However, this approach will not be followed here since, although the Crank-Nicolson scheme enjoys nice properties in terms of entropy dissipation, it does not lead to a {\sl monotone scheme.}

The matrix $B_S$ is symmetric and we can work with few points $N_S$. We summarize our numerical scheme for the Fokker Planck equation as follows : 
\bel{CNFP}
\aligned
S^{n+1/2} &
=S^{n+1} + \sqrt{\frac{A^n \delta^n}{2}}  \eps_{\mathcal{N}}, 
\qquad  S^{n+1} = \Pibf^{n,n+1} S^{n}
\\
 \Pibf^{n,n+1} &  = \exp\big( \delta^n B^{n+1/2} \big), 
 \qquad 
 B^{n+1/2}  
 = -\big( \nabla_{S^{n+1/2}} \big)^T A_{S^{n+1/2}} \big( \nabla_{S^{n+1/2}} \big). 
\endaligned
\ee
{
Once the computation is done for all $0<t^0< \ldots <t^{N_t} = T$, the backward  equation 
reduces to saying 
$P^{n+1} = \Pibf^{n,n+1} P^{n+1}$. 
}


Concerning the fully discrete scheme \eqref{CNFP}, our main result is (unsurprisingly) similar to Claims~\ref{SD2} and \ref{KSD}. 

\begin{claim}[Fully discrete algorithm of the coupled Fokker-Planck-Kolmogorov system]
 \label{SD2D}
Consider a time discretization $t^0<t^{\epsilon}<t^1 \ldots<t^{N_t} =T$ and the full time-discrete numerical scheme above, for the approximation of the transported Fokker-Planck equation \eqref{equa--Lcal}, under the ellipticity 
and linear growth conditions \eqref{equa-assume}.
Then the scheme above defines a fully time-discrete family of moving particles 
$S^{n} =(S^{n}_{m,d})_{m,d} \in \RR^{N_S \times D}$ together with an associated discrete 
measure $\mu_{S^{n}} = \frac{1}{N_S} \sum_{m} \delta_{S^{n_t}_{m}}$,  
which determine a convergent approximation of the solution $\muexa(t^n)$ to \eqref{equa--FP-mod}, as follows. 

\bei
\item {\bf Sup-norm estimate:} 
$
\sup_{ n } e^{-\lambda t^n} \sum_{n=1}^N |S^n|^2 \leq \sum_{n=1}^N  |S^0 |^2.
$

\item {\bf Balance law of momentum:} 
$
 \langle S^n_d,1_{N_S} \rangle_{\ell^2} = \langle S^0_d,1_{N_S} \rangle_{\ell^2}$
 for all $d=1,\ldots, D
$
and, in particular, the total momentum is constant in time when the drift vanishes identically. 

\item {\bf Error estimate.} For any relevant test-function $\varphi$ and any time $t \geq 0$ one has 
\be
\bigg|\int_{\RR^D} \varphi(x) d\muexa(t^n,x) - {1 \over N}  \sum_{n=1}^N  \varphi(S^n)) \bigg| 
\leq 
\dK(\muexa(t^n), \mu_{S^n}) \,  \|\varphi \|_{\Hcal^K(\RD)}. 
\ee

\item {\bf Stochastic property.} The matrix $\Pibf^{n,n+1} \in \RR^{N_S \times N_S}$, describing 
the transition probability 
$\pi_{l,m} = \mathbb{P}(S^{n+1}_m | S^{n}_l)$, is a 
stochastic matrix at each time $t^n$, which is also
 positive provided $(\nabla_S)^T A (\nabla_S)$ is positive (for all $n\ne m$).

\eei

\end{claim}


\paragraph*{Algorithm for the transition probability matrix.}

{
It remains to  compute
 the transition probability matrix \eqref{equa--Kolm}. Our algorithm is based on the following observation. 
{ 
If $t \mapsto X_t \in \RR^D, t \ge 0$ is a martingale process, we fix any integer $N>0$, and any $0 \le s \le t$, and we denote by $x \in \RR^{N\times D}$ (respectively $y \in \RR^{N\times D}$) any i.i.d. sample of $X_t$ (resp. $X_s$).
Then, the operator $\Pi(s,t,x,y)$ is approximated by the following stochastic matrix: 
\be
	\arg \inf_{\Pi \in \mathcal{S}_N} \| y - \Pi x \|_{\ell^2} =  p(x | y) \in \RR^{N \times N}, 
\ee
where $\mathcal{S}_N$ denotes the set of $N \times N$ stochastic matrix.
Our algorithm is decomposed as follow. 
}
\begin{itemize}

\item[] {\it Step 1.} We compute $y   = y - \overline{y}$, $x  = x - \overline{x}$, where for instance $\overline{x} = \frac{1}{N} \sum_{n=1}^N x^i$.

\item[] {\it Step 2.} We solve first 
$
	\epsilon  = \arg \inf_{\epsilon \in \eps_N} \| y - \epsilon x \|_{\Hcal^K}. 
$
This amounts to consider the Linear Sum Assignment Problem (LSAP) problem 
$\overline{\epsilon}  = \arg \inf_{\epsilon \in \eps_N}d_k(\epsilon x, y),
$
where $d_k$ is the pseudo-distance associated with the kernel. 
For the next step, we consider thus $\overline{\epsilon} \, x$ instead of $x$.

\item[] {\it Step 3.} We then deal with  
$
	\arg \inf_{\Pi \in \mathcal{S}_N} \| y - \Pi x \|_{\ell^2}. 
$
To this end, we rely on a gradient descent method and  compute
$
 \frac{d}{dt} \| \Pi x - y \|^2_{\ell^2} = \big \langle \Pi x - y, \frac{d}{dt} \Pi x \big \rangle_{\ell^2}.
$
Observe that the 
{
previous equation does not provide immediate the descent algorithm consistent with the stochastic matrix. However, \eqref{KDS}-\eqref{mart} suggests how to define a {\sl stochastic matrix projection}. Namely, we 
 can introduce the following algorithm in order to approximate the transition probability matrix
}
$ \frac{d}{dt} \Pi 
= - \Big((\nabla_y)^T A  (\nabla_y) \Big)  \Pi, \qquad \Pi(0)  = \Id,
$
where $(\nabla_y^T)^{-1}$ is a discrete operator defined from the kernel.

\end{itemize}
 
}


\subsection{Application to business cases}

\paragraph*{Static hedging strategy.} 

We now describe two particular applications of our method for finance problems, and we present first a static hedging strategy, which can be interpreted as a 
way to get portfolio ``protection'' agains ``negative'' events. 
Our setup will be slightly simplified and we refer to \cite{JMM-SM} for more general business cases.  
We thus consider a process $X_t$ governed by the stochastic equation \eqref{equa-SED} corresponding, for definiteness, to a $D$-dimensional martingale log-normal process modeling forward rate swaps. Such a model can be taken to be the financial model of interest rates of Brace-Gatarek-Musiela \cite{BGM1997}, also referred to as the {LIBOR market model.}  
{ 
From this stochastic process $X_t$ we define a second stochastic process, referred to as 
the \textsl{underlying process} and, for definiteness, we choose the stochastic curves (that is, the
 coefficients of the equation) to be of the following form.
 } 
\begin{itemize}

\item Stochastic discount curves $B(t,s, X_t)$ (defined for $s \ge t$), where $B>0$ is a prescribed nonlinear function.

\item Stochastic swap rates curves $r(t,s, X_t)$, defined by
$
r(t,s, X_t)= \frac{ 1 - B(t,s, X_t) }{B(t,s, X_t)(s-t)}.
$
\end{itemize}
\noindent The portfolio and hedging instruments are then described as follows. 
Given any instrument with payoff $P(t,s,X)$ (where $t$ represents the initial time of the contract), we can 
consider  its fair values $P (t,s,X_t)$ given by the Kolmogorov equation \eqref{equa--KE}, 
or else we consider its sensitivities $\nabla P (t,s,X_t)$ computed as $\nabla_y P (t,y^n(t))_{1 \leq n \leq N}$.
These values are computed by the discretization scheme \eqref{equa-KED}. 
In \cite{JMM-SM} the Economic Value of Equity (EVE) and the Net Income Interest (NII) indicators were chosen as instruments. They are written on top of each loans and assets of the bank account sheet, and they have both swaps rate curves and discount curves as underlyings. 
{ 
NII and EVE describe, with this simple model, {\sl perpetual swaps}. However, they can be turned into barrier options (statistical model) or Bermudan swaptions (rational model), depending on how renegotiation and refund effects are taken into account. 

Fair values and their sensitivities are themselves two stochastic processes of interest in the applications, and we now present some of their numerical simulation.
}
 In Figure~\ref{EVEC}, the left-hand figure represents the expectation of the EVE process as a function of time, while the right-hand figure shows the sensitivities (at time $t=1$) with respect to a particular underlying, called LIB1Y in our model. We considered here the EVE indicator together with a statistical negociation model. 
In this setting, a hedge instrument is also a derivative, having payoff $H_m(t,s,X)$, inside an hedging portfolio of $M$ instruments $H = (H_1,\ldots,H_M)$. For instance we considered swaps, caps, and floors as hedging instruments in \cite{JMM-SM}, but swaptions, even Bermudan ones could be included. 

\begin{Figure}
\includegraphics[width=0.42\linewidth]{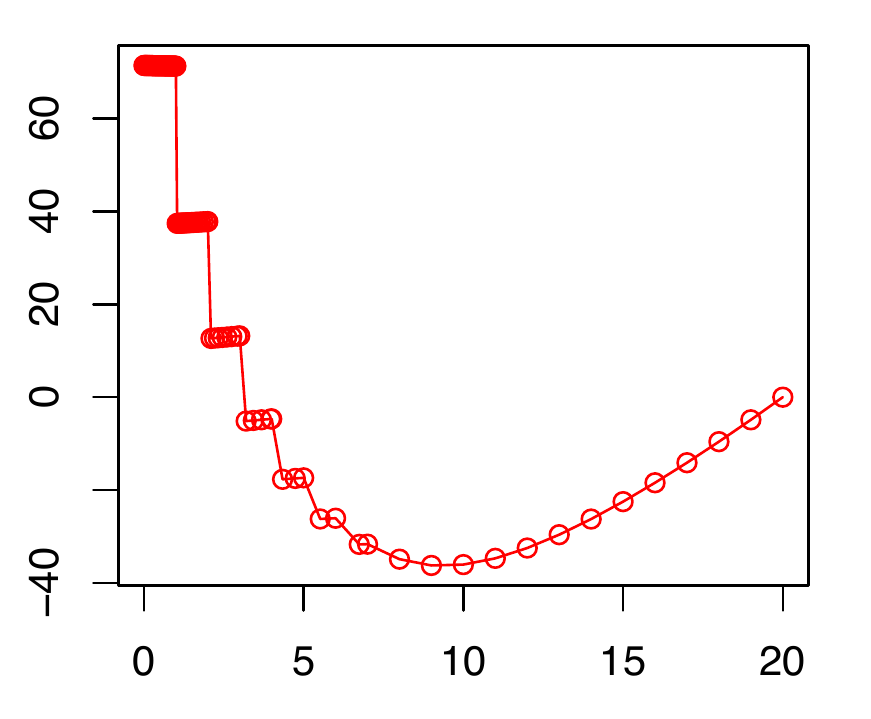} \includegraphics[width=0.42\linewidth]{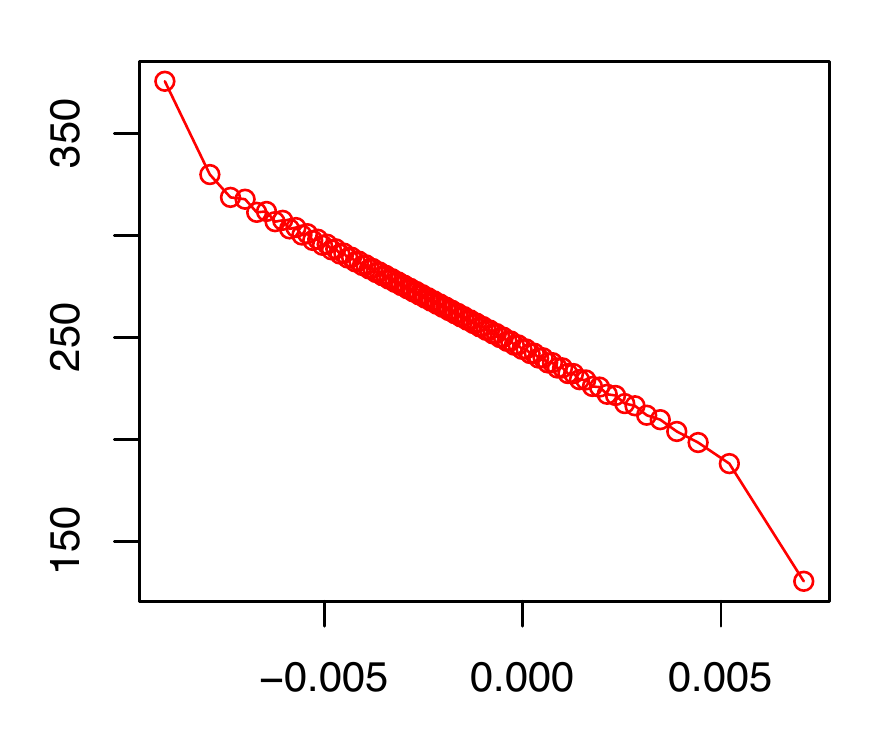} 
%
\centerline{(a): horizontal axis $t$,  vertical axis $E$; \qquad\qquad\qquad (b):  horizontal axis $x^1$, vertical axis $S$}
\captionof{figure}{  (a) expectation. (b) sensitivity at the time $t=1$}
\label{EVEC}
\end{Figure}


\paragraph*{Hedging strategies.}

A {\sl static} hedging strategy, set at the time value $t$, is any function $\alpha(t)=(\alpha_1, \ldots,\alpha_M)(t) \in \RR^M$
 defining a portfolio with payoff $H_\alpha= <\alpha,H>$. 
 { 
 Our algorithm allows us to compute any strategy associated with a criterion such as 
 }
\bel{303}
\overline{\alpha}= \arg \inf_{<\alpha,\overline{H}> \leq C } \quad \int_{s \ge t} 
{\mathbf F} \big([P  - <\alpha,\overline{H}>](t,s,\cdot)\big) \, ds, 
\ee
in which
{ 
$\mathbf F$ is a convex functional and, for instance, can be chosen to be the variance functional 
$\int_{\RR^D} ( P  - P _{mean})^2d\mu$, or the sensitivity functional $ S=\int_{\RR^D} \big | \nabla P \big|^2 d\mu$. 
The restriction $<\alpha,\overline{H}> \leq C$ is included here but need not be imposed, which is certain applications in finance. This constraint imposes a limit on the hedging portfolio investment value.
}
Observe that the hedging strategy \eqref{303} is based on fair values, which are computed in the course of
solving the Kolmogorov equation. On the other hand, with a standard Monte-Carlo approach it would be very costly to
compute such hedges.
In Figure~\ref{EVENIIH} we plot the numerical results for such a hedging strategy. The left-hand figure
shows the variance of the EVE over time, represented in Figure~\ref{EVEC} in the dashed curve, 
while the curve made of circles shows the variance of the EVE after hedging its sensitivities using \eqref{303}. This hedge reduces to almost zero the variance or the sensitivity of the EVE process. However, it also increases dramatically the ones of the NII process; see the right-hand figure.

\begin{Figure}
\includegraphics[width=0.42\linewidth]{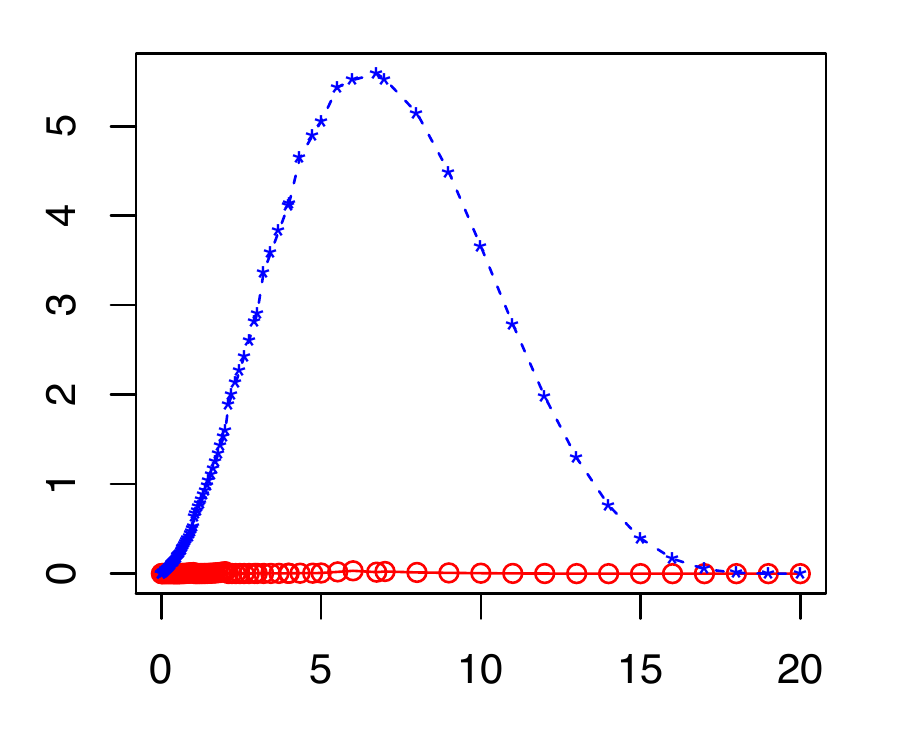} \includegraphics[width=0.42\linewidth]{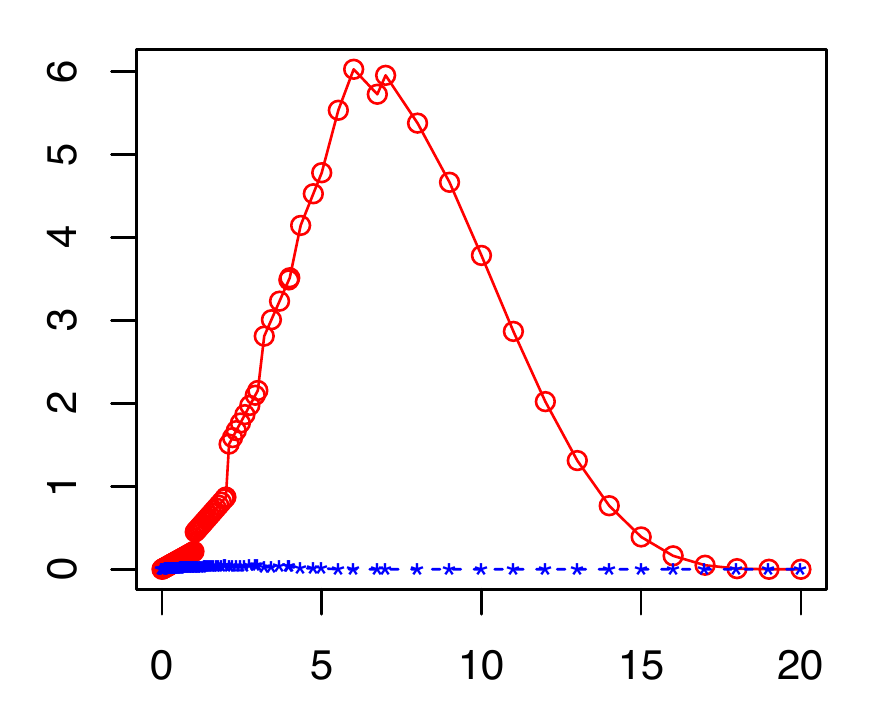} 
\centerline{(a): horizontal axis $t$,  vertical axis $E$ \qquad\qquad\qquad (b):  horizontal  axis $t$, vertical axis $E$}
\captionof{figure}{  (a): expectation EVE. (b): sensitivity EVE}
\label{EVENIIH}
\end{Figure}

\paragraph*{$\beta$-Hedging strategies.}

There are several cases of interest where one would like to mix two different hedging strategies. For instance,
 in the NII / EVE example treated so far, we are using two different hedges computed with \eqref{303}. 
 In the first one, we compute the hedging strategy minimizing the EVE sensitivities, as illustrated in Figure~\ref{EVENIIH}, while in the second one we compute an hedge aiming to reduce the NII sensitivities. 
 This second strategy is very efficient to allows one to lower the NII sensitivities or variance 
 but, as shown in Figure~\ref{EVENIIH}, it increases dramatically the ones of the EVE. 
Now, given any two strategies $\overline{\alpha}^1$ and $\overline{\alpha}^2$ determined from
  \eqref{303}, let us 
consider the new strategy
$
\alpha = \beta \, \overline{\alpha}^1 + (1-\beta) \, \overline{\alpha}^2$ for some $\beta \in [0,1]$ chosen to match some another criterion. For instance, 
in Figure \eqref{BETAEVENIIH} we choose $\beta=0.07$. The figure drawn with star symbols shows the variance of the NII and EVE. The variance associated with the $\beta$-hedge is shown by the curve drawn with circles, and 
shows an interesting reduction in comparison with, both, EVE and NII. 
{ In this numerical example, the value $\beta = 0.07$ is chosen in order to keep the variance of the EVE 
within a certain limit, 
which can be motivated by rules of internal bank or regulatory nature. Within such a limit, our choice was
to minimize the sensitivities of NII.
}

\begin{Figure}
\includegraphics[width=0.42\linewidth]{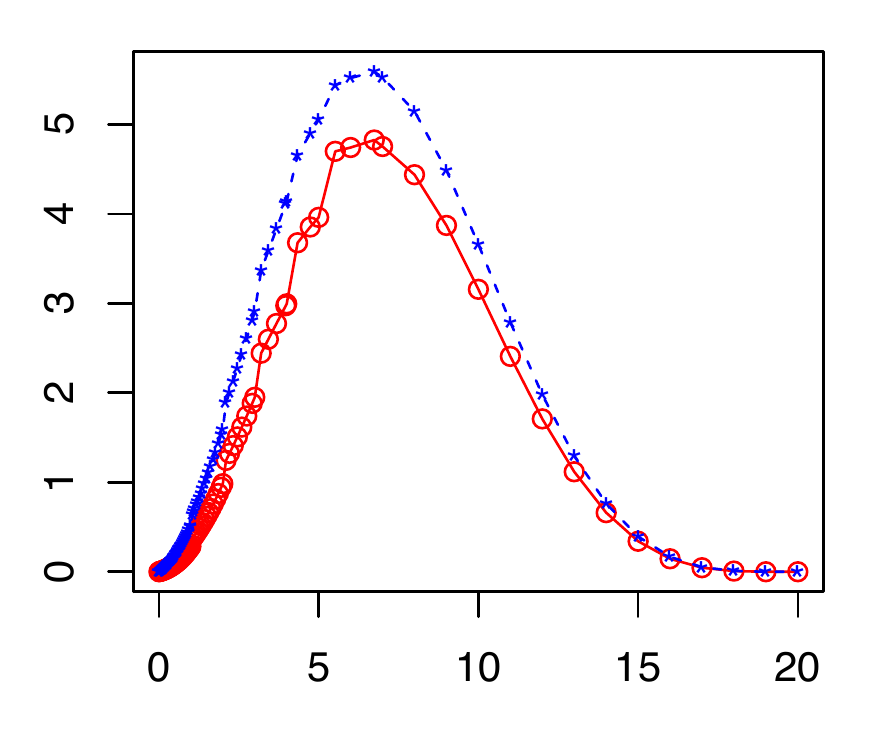} \includegraphics[width=0.42\linewidth]{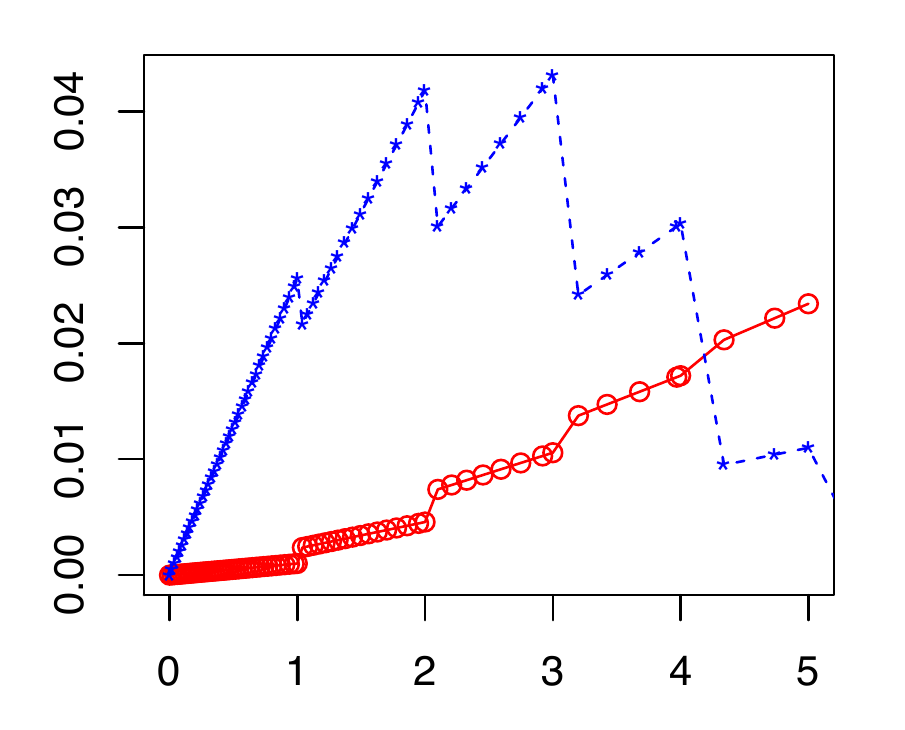} 
\centerline{(a): horizontal axis $t$,  vertical axis $E$ \qquad\qquad\qquad (b):  horizontal  axis $E$, vertical axis $S$}
\captionof{figure}{$\beta$-hedge with $\beta = 0.07$}
\label{BETAEVENIIH}
\end{Figure}


%
 

\section{An algorithm for polar decompositions in transportation theory} 
\label{sec-55}
 
\subsection{Polar and signed-polar decompositions in finite dimensions} 

\paragraph*{Two decompositions of complex numbers.}

The polar decomposition is an important tool in many applications, and we will now present a discrete algorithm for its computation. We begin by a discussion at the continuous level and present several (old and new) notions 
which allow one to decompose matrices and mappings. The polar decomposition of a complex number $z = \rho e^{i\theta} \neq 0$ can be revisited as follow: While it is standard to impose the normalization $\rho > 0$ and $\theta \in [0,2\pi)$, we propose here an alternative decomposition. 
 
\begin{definition} 
The {\em signed-polar decomposition} of a complex number $z= \rho e^{i\theta} \neq 0$, by definition, is 
$z= \rho' e^{i \theta'}$ 
in which one relaxes the signed condition on the modulus while the angle is restricted to remain in the ``half-space'', i.e. 
$\rho' = \pm \rho$ and $\theta' \in (-\pi/2, \pi/2]$. 
\end{definition} 

In other word, from the standard polar decomposition $z= \rho e^{i \theta}$, we  set 
\be
 (\rho', \theta') = 
\begin{cases}
(\rho, \theta), \qquad & \theta \in [0, \pi/2), 
\\
(-\rho, \theta-\pi), \qquad & \theta \in [\pi/2, 3\pi/2),  
\\
(\rho,\theta- 2\pi), \qquad & \theta \in [3\pi/2, 2\pi). 
\end{cases}
\ee
Observe that $(\rho,\theta)$ can be characterized by the property that $\theta$ is a solution to the following minimization problem 
\be
\inf_{\theta \in [0,2\pi)}  |z - e^{i\theta} |^2 
= 
\inf_{\theta \in [0,2\pi)}  |1- e^{-i\theta} z |^2, 
\ee
{ 
while $\theta'$ can be characterized by the minimization problem
}
\be
\inf_{\theta \in (-\pi/2,\pi/2]} |z - e^{i\theta} |^2 
= 
\inf_{\eps = \pm1}  \inf_{\theta \in (-\pi/2,\pi/2]} |\eps - e^{-i\theta} z |^2. 
\ee
In the standard polar decomposition, one performs a projection of $z$ on the unit circle in the complex plane, while in the signed-polar decomposition we perform a projection on the unit half-circle of complex numbers with positive real part.  
We also observe that the mapping $z \neq 0 \mapsto (\rho, \theta)$ is continuous (and even analytic) {\em except} across the positive real line, while the mapping  $z \neq 0 \mapsto (\rho', \theta')$ is continuous (and even analytic) {\em except} across the imaginary line. Interestingly, such a discontinuity property does not arise with real-valued matrices, nor mappings, as we show below. 
Furthermore, it will convenient to slightly modify the above notation. In the following, in dealing with matrices and maps, we will extend the following notation stated here for complex numbers: 
\be
\aligned
& z= \rho^+ e^{i \theta}, \qquad && \rho^+ > 0, \qquad \theta \in [0,2\pi), 
\\
& z = \rho e^{i \theta^+}, \qquad  && \rho \in \RR, \qquad \theta^+ \in (-\pi/2, \pi/2].
\endaligned
\ee 


\paragraph*{Two decompositions of real-valued matrices.}

We next consider matrices and throughout we deal with real-valued $N \times N$ matrices. Recall that the standard {\em polar decomposition} of an invertible matrix $M$ reads 
\bel{PFM}
\aligned
& M = U S^\plus, 
\qquad
 U \text{ is orthogonal,}
\quad
 S^\plus  \text{ is symmetric and positive definite.}
\endaligned
\ee 
Observe that the map $M \mapsto (U,S^\plus)$ is well-defined on the set of invertible matrices, while the decomposition may fail to be unique for non-invertible matrices. In fact, when $M$ is not invertible, its symmetric part $S^\plus = (M^T M)^{1/2}$  is defined uniquely as a semi-positive definite matrix, so that uniqueness fails only for the orthogonal part $U$.
 
In contrast with the decomposition above, in the definition we propose the positivity condition is imposed on the orthogonal part of the decomposition rather than on the symmetric part.  Clearly, the space of symmetric orthogonal matrices plays a particular role here.


\begin{proposition}
[The signed-polar decomposition of matrices]
 \label{PFU} 
1. Any (invertible) matrix $M$ admits a  
decomposition of the form 
\bel{SDM}
\aligned
& M = U^\plus S,  
\qquad
 U^\plus \text{ is orthogonal and positive definite,}
\quad  S  \text{ is symmetric.}
\endaligned
\ee 
referred to as its {\em signed-polar decomposition.} Moreover, the mapping $M \mapsto (U^\plus,S)$ is continuous when restricted to the space of invertible matrices. 

2. The connection between the polar and signed-polar decompositions $M = U^\plus S = U S^\plus$ of an invertible matrix is as follows: there exists a {\em symmetric and orthogonal} matrix denoted by $U_0$ such that
\be
U^\plus  = U U_0, \qquad S = U_0 S^\plus.
\ee
Furthermore, if $S = V D V^T$ denotes the spectral decomposition of $S$ in which $V$ consists of an orthogonal basis of eigenvectors and $D$ is a diagonal matrix made of its real eigenvalues, then one has
\be
S^\plus  = V D^\plus  V^T,  \qquad U_0 = V \Sigma V^T, 
\ee
where $D^\plus$ is defined by taking the absolute value of the elements of $D$ and $\Sigma$ is a diagonal matrix made of  the sign of these eigenvalues.
\end{proposition}

\begin{proof} Any invertible matrix $M$ admits a complex-valued, singular value decomposition 
$
M=W D^\plus V^*,
$
where $W,V$ are unitary matrices and the diagonal matrix $D^\plus$ is {\em real and positive.} Here, $V^*$ stands for the adjoint of the matrix $V$. The standard polar decomposition is obtained immediately by writing
$ 
M= (W V^*) (V D^\plus V^*) = U S^\plus. 
$
For any diagonal matrices $\Sigma$ having $\pm 1$ on its diagonal, we have
$$
	M= (W \Sigma V^*) (V \Sigma D^\plus V^*) = \big( U U_0 \big) \big( U_0^* S^\plus \big), \quad U_0 = V \Sigma V^*
$$
Observe that for any $x \in \RN$ we have 
$
	<(W \Sigma V^*)x, x> =<\Sigma y, W^* V y>, \quad y= V^* x
$
Moreover, any unitary matrix is diagonalizable with orthogonal eigenspaces. In particular if one pick-up $\Sigma$ carrying the signs of $W^* V$, then the previous expression is positive, leading to the signed-polar decomposition:
$
M= (W \Sigma V^*) (V \Sigma D V^*) = U^\plus  S.  
$
\end{proof}


\subsection{Polar and signed-polar decompositions of mappings} 

\paragraph*{Notation.} 

Let us repeat some of our notation. We write $\nabla f = \big(\del_d f \big)_{d=1, \ldots, D}$ for the gradient of a scalar-valued function $f: \RD \mapsto \RR$, 
and we denote by $\nabla \cdot S$ the divergence operator of a map $S: \RD\mapsto \RD$. The Laplace operator
is expressed as $\Delta = \sum_{i=1, \ldots, D} \del_i^2 = \nabla \cdot \nabla$, and 
the Jacobian is writen as $\nabla S = (\del_j S_i)_{i,j}$ (row-oriented). We also use the composition rule
$\nabla (\varphi \circ S) = (\nabla S) (\nabla  \varphi)\circ S$. 

We can work on any open and convex subset $\Lambda \subset \RD$ endowed with the normalized Lebesgue measure 
$m$, with $m(\Lambda) =1$, but in practice we can simply pick up the unit cube $\Lambda = (0,1)^D$, 
and this is what we do in all our numerical tests.  
%
Consider maps $S: y \in \Lambda \mapsto x=S(y) \in \Omega$ taking values in a convex and open subset $\Omega \subset \RD$.
Given a probability measure $\mu$ defined on $\Omega$ and a probability measure $\nu$ defined $\Lambda$, 
a map $S: (\Lambda, \nu) \mapsto (\Omega,\mu)$ is said to be measure-preserving, or to transport $\nu$ into $\mu$, 
provided 
$S_\sharp \nu = \mu$. In other words, one requires the following change of variable formula 
$\int_\Omega\varphi \, \mu = \int_\Lambda  (\varphi \circ S) \, \nu. $ for all test-functions $\varphi$.
Under suitable regularity on the map $S$, this change of variable implies the Jacobian equation
$\left(\mu \circ S \right) |\det \nabla S | = \nu$. 


\paragraph*{The standard polar factorization for mappings.}

By the theory of optimal transport \cite{Villani}, given a positive probability measure $\mu$ defined on $\Omega$, 
(satisfying $\supp \mu = \Omega$ and 
absolutely continuous 
with respect to the Lebesgue measure)
there exists a unique {\sl optimal transport map}
$S: (\Lambda,m) \mapsto (\Omega,\mu)$ transporting the Lebesgue measure $m$ on $\Lambda$ into $\mu$.
It is uniquely characterized by minimizing a suitable ``cost" functional. Precisely, to any map $S:\Lambda \to \Omega$ and probability measure $\mu = S_\# m$, we can associate the decomposition 
\bel{PF}
S= (\nabla h^\plus)\circ T, \qquad h^\plus: \Lambda \to \RR 
\text{ convex}, \qquad T_\#m = m,
\ee
which is called the {\em polar factorization} of $S$.
Hence, any map can be regarded as a gradient of a convex function $\nabla h^\plus $, modulo a Lebesgue measure-preserving map $T$. 
As for the polar decomposition of matrices, let us consider the application $S \mapsto (h^\plus,T)$
which is defined whenever $S$ is invertible in the sense that $\supp S_\#m = \Omega$ and $S_\#m \leq C \, m$. 
The so-called Wasserstein distance between two probability measures $\mu_0$ on $\Omega$ and
$\mu_1$ on $\Lambda$
\bel{WA} 
    d\big(\mu_0,\mu_1 \big) = \Wbf_2(\mu_0,\mu_1 \big) = \inf_{S_\# \mu_1 = \mu_0} \int_\Omega |X-S(X)|_2^2 \, d\mu_0
\ee
allows one to characterize the polar factorization $S= \nabla h: \Lambda \mapsto \Omega$ by solving a minimization problem.


\paragraph*{The signed-polar factorization for mappings.}

As for matrices, we propose here a novel polar factorization for mappings 
$S= \nabla h \circ T$, in which the positivity condition is imposed here on the part $T$ rather than on the convex function $h$.  

\begin{proposition}[Signed polar factorization for maps] \label{SPD}
Let $\Lambda$ be a convex set and $S : \Lambda \to \RD$
be a map  satisfying $\supp S_\#m = \RD$ and $S_\#m \lesssim m$. Then there exists a unique Lebesgue measure-preserving map  $T^\plus$ and a real-valued $h: \Lambda \to \RR$ (which need not be convex) such that 
\be
	S= \nabla h \circ T^\plus, \quad T^\plus_\# m = m, \qquad \nabla T^\plus  \geq 0.
\ee
A connection between signed and polar decomposition is as follows: let $S = (\nabla h^\plus)\circ T$, $h^\plus$ convex, $T_\#m = m$ the standard polar factorization. Then there exists a Lebesgue measure-preserving $T_0= \nabla g$ such that 
\be
	T = T_0 \circ T^\plus, \qquad \nabla h = \nabla h^\plus \circ T_0.
\ee
\end{proposition}

Of course, the signed-polar decomposition is trivial in dimension $D=1$, while the standard polar factorization is non-trivial, corresponding to a re-ordering by increasing values. It generates a Lebesgue measure-preserving field $T$ that is a
 non-smooth re-arrangement of the function. The regularity property of our factorization 
may be viewed as an important advantage in some applications. 
For instance, consider a map $S:\Lambda \mapsto \RN$ with bounded total variation in the sense that its Jacobian is a measure-valued field of matrices. 
Then the signed-polar decomposition allows to give a meaning to $\nabla S = (\nabla T^\plus)(\nabla^2 h)\circ T^\plus$. while the formula $\nabla S = \nabla T (\nabla^2 h^\plus) \circ T$ for the standard polar decomposition is not well-defined.  

\begin{proof} Given an arbitrary mapping $S:\Lambda \mapsto \RD$,
let us consider its Jacobian $\nabla S$. 
Consider the polar and signed-polar decomposition of matrices, say $\nabla S = U^\plus  W = U  W^\plus$, where $W$ defined on $\Lambda$ 
is a square-integrable diagonal field of matrix, and $U^\plus$
is an orthogonal positive field of matrix, hence is bounded. 
Let us introduce the following equations
\bel{169}
\nabla T^\plus  = U^\plus,
\qquad 
\nabla P = W \circ (T^\plus)^{-1}
\ee
Since the polar decomposition is unique for any field $S = (\nabla h) \circ T$, the equation $\nabla T^\plus  = U^\plus$ has a solution for any field $S$ with positive Jacobian, which
 can be expressed as (cf.~the Hodge decomposition in Section~\ref{hodge1}):
$T^\plus  = \Delta^{-1} \nabla \cdot U^\plus 
$. 
Furthermore, $T^\plus$ is Lipschitz continuous
since $\nabla T^\plus$ is orthogonal and positive. Observe that $| \det \nabla T^\plus  | = 1 $, hence we deduce $T:\Lambda \mapsto \Lambda$ is Lebesgue measure-preserving, $T_\#m = m$, and the composition $S \circ (T^\plus)^{-1}$ is meaningful.
In particular, we can solve the second equation as previously, computing explicitly 
$ 
 	P = \Delta^{-1} \nabla \cdot \Big( W \circ (T^\plus)^{-1} \Big).
$
In particular, since $\nabla P$ is a symmetric field of matrix, we must have $P = \nabla h$, for some $h: \Lambda \to \RR$.
Finally, we check that
$ 
	\nabla ((\nabla h) \circ T^\plus) = \nabla T^\plus  (\nabla^2 h) \circ T^\plus  = U^\plus  W = \nabla S,
$
hence, up to a constant, $P\circ T^\plus  = S$. 
Consider next the connection between polar and signed-polar factorizations, 
that is,
 $S= \nabla h \circ T^\plus  = \nabla h^\plus  \circ T$. We deduce $T_0 = T \circ (T^\plus)^{-1} = (\nabla h^\plus)^{-1} \circ \nabla h$ is Lebesgue measure-preserving, and has a symmetric Jacobian field of matrices. Using the same argument as above, we conclude that $T_0 = \nabla g$.
\end{proof}


\subsection{A continuous algorithm for computing polar decompositions}
\label{hodge1}

\paragraph*{The Hodge decomposition.}

Here, $\Lambda =[0,1]^D$ is chosen to be the unit cube (but our analysis would hold for any convex subset of $\RD$) and we denote its boundary by $\Gamma$. 
The {\em Hodge} decomposition consists in decomposing a map $S: \Lambda \to \RD$ 
 as an (for the $L^2$ norm) orthogonal sum of a gradient map plus a divergence-free map. 
 More precisely, there exists a unique harmonic function $h$ on $\Lambda$, 
 a divergence-free map $\zeta$ on $\Lambda$ 
 and a function $h_0$ vanishing on the boundary $\Gamma$ such that 
\be
\label{hodge}
S = \nabla (h + h_0) + \zeta,
\qquad
 \nabla \cdot \zeta  = 0.
\qquad
\zeta \cdot \eta = 0,
\qquad
\Delta h= 0, 
\ee
where $\eta$ is the (outward) normal to $\Gamma$ and $\Delta^{-1}$
is the inverse Laplacian operator with vanishing boundary conditions. Then the components are determined (up to a constant for $h_0$) as ($\mathcal{L}$ being the Leray operator)
\bel{hodge1-d}
   \nabla h_0 = \Pi S = \nabla \Delta^{-1} \nabla \cdot S, 
   \qquad 
   \zeta = \mathcal{L} S = \big( \Id - \nabla \Delta^{-1}  \nabla \cdot  \big) S. 
\ee
%

\paragraph*{Inverse and composition of maps.} 

A mapping $S=S(t,\cdot)$, viewed as an input data, is a priori given (or computed in the applications). 
In particular, it might not define a proper mapping in the sense that $\det \nabla S >0$ might not hold. 
Consider the following equation with unknown $u:\Rd \mapsto \RR$ and prescribed data $v$ 
\bel{LAG}
u \circ S   = v.
\ee
Consider any map $S$
and its polar factorization (see \eqref{PF}), say 
$S = \left(\nabla h\right) \circ T$, with a convex function $h$ and a surjective map $T \in \Acal(\Lambda)$, 
with $S(\Lambda) = \Omega$. Then the following identity 
\bel{S-1}
S^{-1} = T^{-1} \circ \left(\nabla h \right)^{-1}: \RD \mapsto \Lambda
\ee
makes sense provided $\mu = S_\# m$ is dominated by the Lebesgue measure, and this 
leads us to the following solution of \eqref{LAG}: 
\bel{u_b}
 u  = v \circ T^{-1} \circ (\nabla h)^{-1}. 
\ee
To cope with the possibility that $\mu$ has a singular part with respect to the Lebesgue measure $m$, we should use a suitably generalized inverse.


\paragraph*{Lebesgue measure-preserving maps.}

The following lemma shows that any path of divergence free vector generates a time-dependent Lebesgue measure-preserving maps. 

\begin{lemma}[Time-dependent Lebesgue measure-preserving maps] 
\label{TDLPM}
Consider a time-dependent path of Lebesgue measure-preserving map $T=T(t, \cdot)$ defined on $\Lambda$
with $T(t,\cdot)_\#m = m$. Then there exists a time-dependent divergence-free map $\chi=\chi(t, \cdot)$ on $\Lambda$ 
such that
    $\del_t T  = \chi \circ T$. 
\end{lemma}

\begin{proof} 
If $T(t,\cdot)_\#m = m$ is Lebesgue measure-preserving, it is
a one-to-one map from $\Lambda$ into itself and we can define $T^{-1}(t,\cdot)$ pointwise. 
For any smooth function $\varphi$, using the map $y = T(t,z)$, we have $\int_\Lambda  \varphi \circ T = \int_\Lambda  \varphi$, and we find 
$$
    0 = \frac{d}{dt} \int_\Lambda  \varphi = \frac{d}{dt} \int_\Lambda  \varphi \circ T  = \int_\Lambda  (\nabla \varphi) \circ T \cdot \del_t T = \int_\Lambda  (\nabla \varphi) \cdot (\del_t T) \circ T^{-1},
$$
showing that $(\del_t T)\circ T^{-1}$ is divergence-free. Conversely, assuming $T$ satisfying $ \del_t T  = \chi \circ T$, 
then it is now clear by the same computation that $T$ is Lebesgue measure-preserving.
\end{proof} 


\paragraph*{A direct algorithm for polar factorizations.}

We arrive at the formulation of our continuous method for computing the polar factorization, 
using a connection between factorization of matrix and maps.

\begin{proposition}[Factorization of smooth positive maps] 
\label{OTM2}
Let $S: \Lambda \to \RD$
be smooth map with convex support $\supp S_\#m$ and positive Jacobian $\nabla S >0$. 
Consider its polar factorization $S = (\nabla h) \circ T$. Then, up to a constant, one has 
\be
	 T = \Delta^{-1} \nabla \cdot U, \qquad \nabla h = \Delta^{-1} \nabla \cdot  (P \circ T^{-1}), 
\ee
where $(P,U^+): \Lambda \mapsto \mathcal{M}(\RR^{D \times D})^2$ is 
the field of symmetric and orthogonal matrices generated by the polar decomposition of matrices $\nabla S = U P$. 
\end{proposition}

\begin{proof} We start from the polar factorization $S = (\nabla h) \circ T$, thus
 $\nabla S = \nabla T (\nabla^2 h) \circ T$. 
Cconsider the Jacobian $\nabla S$
and its polar decomposition in \eqref{PFM}, thus $\nabla S = U P$, where $P$ is a square-integrable field of symmetric positive matrix and $U$ is an orthogonal positive field of matrix.
Since the polar factorization $S = (\nabla h) \circ T$ is unique, we can identify both parts as follows: 
\bel{169}
 	\nabla T = U,\qquad \nabla^2 h = D \circ T^{-1}.
\ee
Consider the equation $\nabla T = U$ and recall that the Hodge decomposition 
yields the unique candidate to be a solution:
$
 	T = \Delta^{-1} \nabla \cdot U,
$
which is Lipschitz continuous. 
Since $T$ is unitary, we can consider its inverse $T^{-1}$ and the composition $D \circ T^{-1}$ is meaningful.
In particular, we can solve the second equation as previously, by computing explicitly
$$
\nabla h = \Delta^{-1} \nabla \cdot \big( D \circ T^{-1} \big),
$$
leading to a solution $h: \Lambda \to \RR$. 
Finally, we check that 
$
	\nabla ((\nabla h) \circ T) = \nabla T (\nabla^2 h) \circ T = U D = \nabla S 
$
and, up to a constant, $P\circ T = S$.  
\end{proof}


\paragraph*{A steepest descent algorithm for polar factorizations.}

{

Proposition~\ref{OTM2} now suggests a direct method in order to compute the polar factorization. As for polar factorization of matrices, we introduce a steepest descent algorithm, as follows.
}

\begin{proposition}[Polar factorization of smooth convex maps via steepest descent] 
\label{OTM}
Given a smooth map $S$ with positive Jacobian matrix $\nabla S \geq 0$, consider the  minimization problem
associated with the Leray operator $\Lcal$
\bel{GeneralWasserstein}
 \Tbar = \arg \min_{ T_\# m = m  } \int_\Lambda  | \mathcal{L} (S \circ T)|^2 \, m, 
\ee
which provides one with the
factorization $S \circ \Tbar = \nabla h$.
%
The steepest descent algorithm for this  problem amounts to compute a family of 
Lebesgue measure-preserving maps $t \mapsto T(t,\cdot)$ with positive Jacobian $\nabla T \geq 0$ satisfying\footnote{ For convex maps, the choice $\del_t T^{-1} = \zeta \circ T^{-1}$ is an alternate and simpler computational choice.}
\bel{OMA}
\aligned
& S \circ T = \nabla h + \zeta, \qquad\qquad\qquad \nabla \cdot \zeta = 0
\\
& \del_t T^{-1} = \mathcal{L} \big(\nabla^2 h \zeta \big) \circ T^{-1}, 
\qquad \del_t T = - (\nabla T) \mathcal{L} \big(\nabla^2 h \zeta \big),
\endaligned
\ee
where $\big(h, \zeta \big)(t,\cdot)$ is obtained via the Hodge decomposition \eqref{hodge} in $\Lambda$, 
and $T^{-1}(t,\cdot)$ denotes the inverse of $T$ and $\mathcal{L}$ the Leray operator. 
When $\supp S = \Omega$, consider a path $T=T(t,\cdot)$ of Lebesgue measure-preserving maps satisfying \eqref{OMA}. Then, as $t \to +\infty$, the map $t \to S \circ T(t,\cdot)$ converges strongly at an 
exponential rate toward the signed-polar factorization of $\overline{S}$:
\bel{equa-PFA}
  S \circ T(t,\cdot) \to \overline{S} = \nabla \phi,
  \qquad
  \phi: \Lambda \to \RR 
  \text{ is convex, as } t \to +\infty.
\ee
\end{proposition}

\begin{proof}  
We write $\del_t T^{-1} \circ T = \chi$, where (by Lemma~\ref{TDLPM}) the map 
$\chi(t,\cdot)$ is divergence-free.  Considering the Hodge decomposition $S\circ T = \nabla h + \zeta$, where $\zeta(t,\cdot)$ is divergence free, we find 
$ 
\frac{d}{dt} \int_\Lambda | \zeta |_2^2 
  = \int_\Lambda <\nabla h \circ T^{-1} - S\cdot (\nabla^2 h) \circ T^{-1} \del_t T^{-1}>.
$
Composing back with $T$, we obtain 
$$
	\int_\Lambda < \nabla h - S\circ T  \cdot (\nabla^2 h)  (\del_t T^{-1})\circ T> = - \int_\Lambda <(\nabla^2 h) \zeta \cdot  \chi>, 
$$
thus the steepest descent algorithm reduces to $\chi(t,\cdot) = \mathcal{L} \big((\nabla^2 h) \zeta(t,\cdot) \big)$. 
\end{proof}
 
  
\subsection{A discrete algorithm for computing the polar decomposition}

\paragraph*{Computing the polar factorization.}

{
 
We now describe our discrete algorithm which generates a polar factorization, along the lines suggested by our theoretical study. From the numerical standpoint, our algorithm takes as an input any two distributions
 of the same dimension $x \in \RR^{N \times D}$ and $y \in \RR^{N \times D}$ and generates, as an output, 
 $z \in \RR^{N \times D}$ obtained by solving the discrete counterpart to \eqref{equa-PFA}, namely: 
\bel{PAF}
   y = \nabla_z h,
    \qquad (h,z) = \arg \min_{z \in \RR^{N \times D}, h \in \RR^{N} } \| y - \nabla_z h \|^2_{\ell^2}
\ee
This minimization problem relies on the Helmholtz-Hodge decomposition, which is computed using a kernel   
(cf.~\cite{PLF-JMM-SM-Tutorial} for further details). 
We emphasize that this algorithm assumes that the initial map is positive in the sense that 
its Jacobian is positive, that is, $\nabla_x h \ge 0$. To this purpose, a reordering algorithm is required
(and is described in \cite{PLF-JMM-SM-akra}). In particular, the operator $(\nabla_x)^{-1}$ is presented in \cite{PLF-JMM-SM-Tutorial}. 

}

\

{ 

{\bf Algorithm 1: computation of $y = \nabla_x h$ with a convex $h$.} 
 An admissible kernel $K$ and a threshold parameter $\epsilon>0$ are given, together with 
two distributions of points $x \in \RR^{N \times D}$ and $y \in \RR^{N \times D}$.
\bei

\item Update the variable $x$ to now be $x^\sigma$, where $\sigma = LSAP(d_k(x,y))$ denotes
 the permutation computed by the linear sum assignment problem (LSAP) described in \cite{PLF-JMM-SM-akra} and $d_K$ denotes the discrepancy error matrix.

\item As long as $\| y - \nabla_x h \|_{\ell^2} < \epsilon$ holds, perform {\sl iteratively}
the following steps: 
\bei

\item Compute the Helmholtz-Hodge decomposition $h  = (\nabla_x)^{-1} (y) \in \RR^{N}$
and  $\zeta  = y - \nabla_x h \in \RR^{N \times D}$.

\item Compute $\lambda = \arg \min_{\lambda > 0 } \| y - (\nabla_x h)(w) \|^2_{\ell^2}$
and  $w  = x - \lambda \zeta$.

\item Update $x$ to be now $x - \lambda \zeta$, which provides us with the numerical values of $x$ and $h$. 

\eei
 
\eei

}


{ 

\paragraph*{The sampling technique.}

We illustrate the use of the polar factorization by presenting a function allowing to recover, from any input distribution $x \in \RR^{N_x \times D}$ and any integer $M>0$, a new 
distribution $z \in \RR^{M \times D}$ which has very similar statistical properties to the initial distribution $x$. 
We can use this sampling algorithm in a number of practical situations, as it allows us to
 produce test data in order to check other algorithms. 
 This algorithm is described as follows. 

\

{\bf Algorithm 2: computation of $M$ i.i.d. samples from an input distribution.}
A positive definite kernel $K$ and a threshold parameter $\epsilon>0$ are given, 
together with a distribution $y \in \RR^{N_x \times D}$ and an integer $M$.

\bei

\item Compute $x \in [0,1]^{N_x \times D}$, that is, take an i.i.d. sample of the uniform law on the unit cube.

\item Compute $z \in [0,1]^{N_x \times D}$ and $h \in [0,1]^{N_x}$ as
the polar factorization $y = \nabla_z h$, $h$ convex, by using Algorithm~1.

\item Compute $w_M \in [0,1]^{M \times D}$, by taking an i.i.d. sample of the uniform law on the unit cube.
%
This provides us with the numerical values for $(\nabla_z h)(w_M)$.

\eei 

}
 

\begin{figure}
\includegraphics[width=.9\linewidth, height=.4\linewidth]{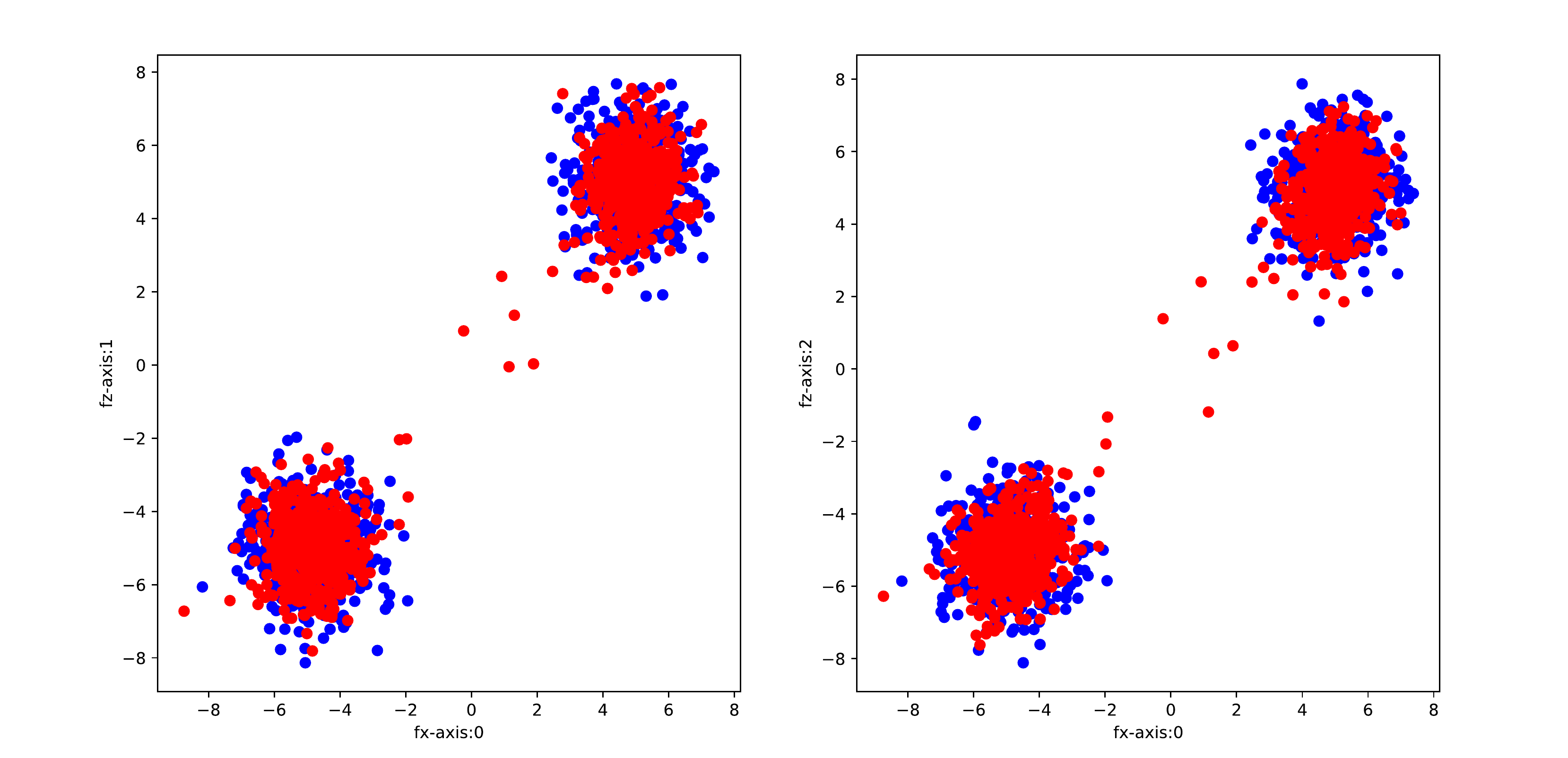} 
\centerline{(a): horizontal axis $e_1$,  vertical axis $e_2$\qquad\qquad (b):  horizontal  axis $e_1$, vertical axis $e_3$\qquad\qquad\qquad}
\caption{  (a): original distribution (blue color); (b): output distribution (red color).
}
\label{fig-ksjd22}
\end{figure}


\begin{longtable}[]{@{}rrrrrrr@{}}
\tabularnewline
\toprule
distribution& skew0 & kurtosis0 & skew1 & kurtosis1 & skew2 & kurtosis2\tabularnewline
\midrule
\endfirsthead
\toprule
distribution& skew0 & kurtosis0 & skew1 & kurtosis1 & skew2 & kurtosis2\tabularnewline
\midrule
\endhead
original & 0.0028489 & -1.847989 & 0.0024598 & -1.844245 & -0.0022913 &
-1.839244\tabularnewline
generated &0.0096915 & -1.882242 & 0.0191324 & -1.865481 & 0.0172387 &
-1.879981\tabularnewline
\bottomrule
\caption{Skewness and kurtosis on each axis}
\label{Table51}
\end{longtable}

\begin{longtable}[]{@{}rr@{}}
\tabularnewline
\toprule
statistic & pvalue\tabularnewline
\midrule
\endfirsthead
\toprule
statistic & pvalue\tabularnewline
\midrule
\endhead
0.048 & 0.1995737\tabularnewline
0.034 & 0.6101665\tabularnewline
0.037 & 0.5005674\tabularnewline
\bottomrule
\caption{Kolmogorov-Smirnoff test}
\label{Table52}
\end{longtable}


{
An an illustration, let us choose $N_x=1000$ and consider a bi-modal distribution, defined by
\be
	x^{2n} \text{ i.i.d. } \mathcal{N}(\mu^+,I_d), \qquad x^{2n+1} \text{ i.i.d. } \mathcal{N}(\mu^-,I_d), 
\ee
where $\mathcal{N}(\mu^\pm,I_d)$ denotes the normal law centered 
at $\mu^\pm  = (\pm 5,\ldots,\pm 5) \in \RR^D$ with unit variance.
In Figure~\ref{fig-ksjd22}, we plot the original distribution together with the numerical output distribution. Since these are three-dimensional distributions, we plot their marginal on the combinations of axis $(e_1,e_2)$ and $(e_1,e_3)$.

We then perform several tests. On one hand, we consider the marginal distributions of $x$ and $z$ on each axis. 
On the other hand, in order to also provide 
some information on the joint distribution, we compute the discrepancy error between the two distributions. Observe that no other general, statistical indicator is available that would
allow one to compare two probability distributions. 
Skewness is a measure of the degree of symmetry, while 
kurtosis is a measure of the concentration of the distribution.  
}
%
%
Skewness and kurtosis comparison tests on each axis; cf.~Table~\ref{Table51}. 
%
Kolmogorov-Smirnov tests on each axis shows that we can not disprove the hypothesis that two independent samples are drawn from the same continuous distribution; cf.~Table~\ref{Table52}. 
%
We also computed the discrepancy error, namely $d_k(\mu_X,\mu_Z) = 0.0812096$, which 
shows that the two distributions agree very well.


\clearpage

\section{Concluding remarks}

{ 

In summary, we presented a mesh-free numerical strategy which has led us to numerical algorithms enjoying the following features. 

\bei 

\item Once a kernel is chosen, the discretization of the relevant solution operators can be performed automatically, via the kernel-based discretization formulas presented in Section 2. There is no limitation to the type of operators, continuous or discrete, that we can cover. The corresponding formula are easy to implement in a code, as we are currently doing in CodPy \cite{PLF-JMM-SM-long}. 

\item The algorithms are endowed with quantitative error estimates which are based on the notion of discrepancy error. Our strategy indeed 
allows us to evaluate the relevance and accuracy of, both, the given data and the numerical solutions, which is an important property in the industrial applications. For instance, see \cite{PLF-JMM-review} for an application to mathematical finance. 
\eei

\bei 

\item Among other applications we presented, a novel application was proposed in Section 5 where we arrived at 
a new algorithm for the standard polar factorization as well as for a new, signed, polar factorization. 
We illustrated this algorithm with some preliminary numerical tests which demonstrate its interest for applications. 

\eei

In future work, it will interesting to establish rigorous convergence proofs for the class of numerical algorithms proposed in the present paper. While some considerations should be general in nature and apply to any given operator, such as the derivation of error estimates, we should also seek to establish model-dependent estimates,
 such as the ones in Section 4 for the Fokker-Planck-Kolmogorov system where we 
 took advantage of the specific structure of the model into account. We plan to continue the development of the CodPy code 
 in the future, and eventually make it available to the scientific community for academic as well as industrial applications. 
 }


\end{document}